\documentclass[11pt]{article}
\usepackage{amsthm, amsmath, amssymb, amsfonts, url, booktabs, tikz, setspace, fancyhdr, amsbsy}
\usepackage[margin = 0.8in]{geometry}
\usepackage{hyperref, enumerate}
\usepackage{caption,float}
\usepackage{subcaption}
\usepackage{esint}
\usepackage{verbatim}
\usepackage{multirow}
\usepackage[title]{appendix}%

\newtheorem{theorem}{Theorem}[section]
\newtheorem{assumption}[theorem]{Assumption}
\newtheorem{remark}[theorem]{Remark}
\newtheorem{proposition}[theorem]{Proposition}
\newtheorem{lemma}[theorem]{Lemma}
\newtheorem{corollary}[theorem]{Corollary}

\theoremstyle{definition}

\newcommand{\R}{\mathbb{R}}
\newcommand{\N}{\mathbb{N}}

\newcommand{\defeq}{\mathrel{\mathop:}=}

\newcommand{\uu}{\boldsymbol{u}}

\newcommand{\vv}{\boldsymbol{v}}

\newcommand{\ww}{\boldsymbol{w}}

\newcommand{\fuu}{\boldsymbol{\mathfrak{u}}}
\newcommand{\fu}{\mathfrak{u}}
\newcommand{\fmm}{\boldsymbol{\mathfrak{m}}}
\newcommand{\fm}{\mathfrak{m}}

\newcommand{\dx}{\,\mathrm{d}\boldsymbol{x}}
\newcommand{\dy}{\,\mathrm{d}\boldsymbol{y}}
\newcommand{\dt}{\,\mathrm{d}t}

\newcommand{\dv}{\,\mathrm{d}\boldsymbol{v}}
\newcommand{\PPhi}{\boldsymbol{\Phi}}
\newcommand{\xx}{\boldsymbol{x}}
\newcommand{\yy}{\boldsymbol{y}}
\newcommand{\bb}{\boldsymbol{b}}
\newcommand{\ff}{\boldsymbol{d}}
\newcommand{\om}{\omega}

\numberwithin{equation}{section}

\def\ocirc#1{\ifmmode\setbox0=\hbox{$#1$}\dimen0=\ht0
    \advance\dimen0 by1pt\rlap{\hbox to\wd0{\hss\raise\dimen0
    \hbox{\hskip.2em$\scriptscriptstyle\circ$}\hss}}#1\else
    {\accent"17 #1}\fi}

\begin{document}

\title{Quasi-Monte Carlo finite element approximation of the Navier--Stokes equations with random initial data}

\author{Seungchan Ko\thanks{Department of Mathematics, Inha University, Incheon, Republic of Korea. Email: \tt{scko@inha.ac.kr}}, ~Guanglian Li\thanks{Department of Mathematics, The University of Hong Kong, Pokfulam Road, Hong Kong. Email: {\tt{lotusli@maths.hku.hk}} GL acknowledges the support from Young Scientists fund (Project number: 12101520) by NSFC and GRF (project number: 17317122), RGC, Hong Kong.}
,~and ~Yi Yu\thanks{School of Mathematics and Information Science, Guangxi University, Nanning, Guangxi, PR China. Email: \tt{yiyu@gxu.edu.cn}}}

\date{~}

\maketitle

~\vspace{-1.5cm}

\begin{abstract}
In this work, we investigate the numerical approximation of the Navier--Stokes equations on a bounded polygonal domain in $\mathbb{R}^2$, where the initial condition is modeled by a log-uniform random field. We propose a novel numerical scheme to compute the expected value of linear functionals of the solution to the Navier--Stokes equations. The scheme is based on the finite element discretization in space, backward Euler method in time, truncated Karhunen--Lo\`eve expansion for realizing the random initial condition, and lattice-based quasi-Monte Carlo (QMC) for estimating the expected values over the parameter space (i.e., randomly-shifted lattice rules for the integration over a high-dimensional hypercube). We present a rigorous error analysis of the numerical scheme, including bounds on the mean squared error, and especially establish that the QMC converges optimally at an almost-linear rate with the constant independent of the dimension of integration. This provides one rigorous theoretical justification of the use of the QMC sampling strategy for quantifying the uncertainties of nonlinear PDEs subject to random inputs.
\end{abstract}

\noindent{\textbf{Keywords:} quasi-Monte Carlo method, finite element method, uncertainty quantification, Navier--Stokes equations, random initial data, Karhunen--Lo\`eve expansion}

\smallskip

\noindent{\textbf{AMS Classification:} 65D40, 65D32, 65N30, 76D05}

\section{Introduction}\label{sec_intro}
Uncertainty in the input data of mathematical models has received much recent attention due to its importance in various applications. The uncertainty can arise from different sources, e.g., coefficients, boundary conditions, initial conditions, and external forces. The impact of uncertainty on various quantities of interest sheds valuable insights into the inherent variability of diverse physical and engineering problems.  To describe and analyze uncertainty, probability theory offers a flexible framework where uncertain inputs are treated as random fields, and it is especially useful in characterizing the randomness associated with physical quantities of a given system.

In this work, we investigate the incompressible Navier-Stokes equations with random initial data.
Let $D\subset\mathbb{R}^2 $ be a bounded convex polygonal domain. Let $(\Omega, \mathcal{F}, \mathbb{P})$ be a probability space,
where $\Omega$ is a sample space consisting of all possible outcomes, $\mathcal{F}\subset2^{\Omega}$  a $\sigma$-algebra, and $\mathbb{P}:\mathcal{F}\rightarrow[0,1]$ a probability function. Consider the following initial-boundary value problem: find a random velocity field $\uu: [0,T]\times D\times\Omega\rightarrow\mathbb{R}^2$ and a random pressure field $p: [0,T]\times D\times\Omega\rightarrow\mathbb{R}$ such that the following equations hold $\mathbb{P}$-almost surely (a.s.):
\begin{alignat}{2}
\partial_t\uu+(\uu\cdot \nabla)\uu-\Delta\uu&=-\nabla p,\quad &&\mbox{in $(0,T]\times D\times\Omega$},\label{main_eq1}\\
\text{div}\,{\uu}&=0,\quad &&\mbox{in $(0,T]\times D\times\Omega$},\label{main_eq2}
\end{alignat}
where the symbols $\nabla$ and $\Delta$ denote differential operators with respect to the spatial variable $\xx\in D$, and $\partial_t$ denotes the time derivative. The system is subject to the homogeneous Dirichlet boundary condition and random initial condition:
\begin{align}\label{main_bc}
\uu(t,\xx,\om)&=\boldsymbol{0},\quad{\rm{on}}\,\,(0,T]\times\partial D\times\Omega,\\
\label{main_ic}
\uu(0,\xx,\om)&=\uu^0(\xx,\om),\quad{\rm{in}}\,\,D\times\Omega.
\end{align}
In this work, we consider the following initial random field $\uu^0(\xx,\omega): D\times\Omega\rightarrow\mathbb{R}$:
\begin{equation}\label{lognormal_initial}
\uu^0(\xx,\omega)=\nabla^{\perp}\exp(Z(\xx,\omega))=\left(-\partial_{x_2}\exp(Z(\xx,\omega)),\partial_{x_1}\exp(Z(\xx,\omega))\right),
\end{equation}
where $Z(\cdot,\cdot)\in L^2(\Omega;L^2(D))$ is a centered Gaussian random field. The choice
\eqref{lognormal_initial} ensures that the
initial condition $\uu^0$ is divergence-free, i.e.,
${\rm{div}}\,\uu^0(\xx,\om)=0$ in $D\times\Omega$.
The quantity of interest is the expected value of $\mathbb{E}[\mathcal{G}(\uu)]$ for any linear functional
$\mathcal{G}\in(L^2(D)^2)'$, associated with the solution $\uu$ of problem \eqref{main_eq1}-\eqref{lognormal_initial},
where $(L^2(D)^2)'$ refers to the dual space of $L^2(D)^2$.

We shall develop and analyze a novel algorithm for computing the expected value $\mathcal{G}(\uu)$.
The proposed scheme consists of three steps. The first step involves solving the problem
\eqref{main_eq1}-\eqref{main_ic} for a fixed $\omega\in\Omega$ using a fully-discrete finite element
method (FEM) scheme, based on the backward Euler in time (with a time step size $\tau$) and the
conforming finite element approximation on a regular mesh of the domain $D$ (with a mesh size $h$).
We denote the finite element solution at time step $J$ by $\uu^J_{h}$. The second step involves
truncating the Karhunen-Lo\`{e}ve (KL) expansion \cite{KL_1, KL_2} of the random initial data
$\uu^0(\xx,\omega)$ in \eqref{lognormal_initial}, which parameterizes $\uu^0$ by a countable number
of random variables, cf. \eqref{KL} below: We truncate the infinite series in \eqref{KL} with a sum
with $s$ terms, $s\in\mathbb{N}$, and obtain a truncated problem with a finite-dimensional random
vector $\yy=(y_1(\omega),\ldots,y_s(\omega))$. We
denote  by $\uu^J_{s,h}$ the FEM solution of the truncated problem. In
the third step, we approximate the quantity of interest $\mathbb{E}[\mathcal{G}(\uu(t_J))]$ by the
expected value $\mathbb{E}[\mathcal{G}(\uu^J_{s,h}(\omega))]$, and compute the integral
using quasi-Monte Carlo (QMC) methods \cite{survey_1, survey_2, survey_3}. Let $\varphi(y)=
\exp(-y^2/2)/\sqrt{2\pi}$ be the probability density function (PDF) of the standard normal distribution, and $\PPhi_s$ the cumulative
distribution function (CDF) of the standard normal random vector of length $s\in\mathbb{N}$, and $\PPhi^{-1}_s(\vv)$ its inverse for $\vv\in[0,1]^{s}$.
Then with $F^J_{s,h}(\yy)=\mathcal{G}(\uu^J_{s,h}(\cdot,\yy))$,
\begin{equation}
\mathbb{E}[\mathcal{G}(\uu^J_{s,h})]
=\int_{\R^s}\mathcal{G}(\uu^J_{s,h}(\cdot,\yy))\prod_{i=1}^s\varphi(y_i)\dy
=\mathbb{E}[F^J_{s,h}]=\int_{(0,1)^{s}}F^J_{s,h}(\PPhi^{-1}_s(\vv))\dv. \label{goal_quant}
\end{equation}
%When $s=1$, we abbreviate the notation as $\Phi:=\PPhi_1$.
To approximate \eqref{goal_quant}, we use the QMC quadrature rule, i.e., {randomly-shifted lattice rule  (RSLRs)}:
\begin{equation*}
    \mathcal{Q}_{s,N}(F^J_{s,h};\boldsymbol{\Delta})
\defeq\frac{1}{N}\sum^N_{i=1}F^J_{s,h}
\left(\boldsymbol{\Phi}^{-1}_s\left({\rm{frac}}
\left(\frac{i\boldsymbol{z}}{N}+{\boldsymbol{\Delta}}\right)\right)\right),
\end{equation*}
where $\boldsymbol{z}\in\mathbb{N}^{s}$ is a generating vector, $\boldsymbol{\Delta}$
is a random shift uniformly distributed over $[0,1]^{s}$, and the function
$\mathrm{frac}(\cdot)$ takes the fractional part of its argument.

The design and analysis of QMC methods for linear PDEs with random coefficients,
where the QMC methods are utilized to estimate the expected values of linear functionals of the 
exact or approximate solution from the PDEs, has been well investigated in the past ten years 
\cite{KSS2012,main_ref, QMC_multi_4,Harbrecht2017}. The expected values can be transformed 
into an infinite or high dimensional integral over the parameter domain, which corresponds to 
the randomness induced by the random coefficient. A proper function space for a given integrand 
should be designed to guarantee that the associated high-dimensional integration problem is tractable.
The "standard" function spaces in the QMC analysis are weighted Sobolev spaces on $(0,1)^s$, 
which consists of functions with square-integrable mixed first derivatives. If the integrand 
lies in a suitable weighted Sobolev space, then RSLRs can be constructed so as to (nearly) achieve 
the optimal convergence rate $\mathcal{O}(n^{-1})$; see \cite{qmc_ref_1, qmc_ref_2, qmc_ref_3, qmc_ref_4, qmc_ref_5} 
and recent surveys \cite{qmc_ref_6, qmc_ref_7}.

In this work, we provide a thorough analysis of the numerical approximation $\mathcal{Q}_{s,N}
(F^J_{s,h};\boldsymbol{\Delta})$. This is achieved by decomposing
the error into the finite element error, the error due to KL
truncation, and the error due to the QMC quadrature. In Theorem \ref{final_main_thm}, we establish a bound
on the root-mean-square error at each time level $J$, i.e.,
$(\mathbb{E}^{\boldsymbol{\Delta}}[(\mathbb{E}[\mathcal{G}(\uu(t_J))]-\mathcal{Q}_{s,N}(F^J_{s,h};\boldsymbol{\Delta}))^2])^{1/2}$,
where $\mathbb{E}^{\boldsymbol{\Delta}}$ denotes taking expectation with respect to the random shift $\boldsymbol{\Delta}$. The main technical
challenge in the analysis is to bound the mixed derivatives of the velocity field with respect to the stochastic variable in suitable weighted Sobolev norms. This issue is highly challenging due to the nonlinearity of the mathematical model. These estimates are established in Theorem \ref{reg_estimate}, which allows deriving an error estimate for the QMC sampling strategy applied to the Navier--Stokes equations. 

Now we situate the study within existing works on QMC for nonlinear problems. So far despite the extensive studies on QMC for uncertainty quantification, most theoretical analyses on QMC methods are confined to uncertainty quantification for linear PDEs. Further, Harbrecht et al. \cite{HarbrechtSiebenmorgen:2016} studied the computing of nonlinear quantities of interest for a linear PDE with log-normally distributed random coefficients using a deterministic QMC rule. In sharp contrast, the study on nonlinear PDEs is fairly scarce, largely due to a striking lack of study on the related solution regularity theory. In an interesting piece of work, Cohen et al. \cite{holo_2018} also investigate the stationary Navier--Stokes equations with random parameters, and establish the holomorphic dependence of the velocity and the pressure on the domain, with the parameterized families of domain mappings over the uniform distribution. When compared with the work \cite{holo_2018}, we investigate the non-stationary Navier-Stokes equations with random initial conditions, and established new mixed derivative estimates for the solutions. Very recently, Harbrecht et al. \cite{gevrey-semilinear2023} establish regularity bounds for a parametric semilinear elliptic PDE with the nonlinearity appearing on the lowest order term. This result was shown under an $s$-Gevrey assumption using the implicit mapping theorem. This analysis technique differs markedly from the strategy in this work. This work contributes new regularity estimates, and furnishes compelling evidence that QMC methods have significant potential for the uncertainty quantification of nonlinear PDEs. 

The rest of the paper is organized as follows. In Section \ref{sec_pre}, we describe the variational formulation of the Navier--Stokes equations and its well-posedness, and give assumptions on the initial data $\uu^0$. In Section \ref{sec_fem}, we give a fully discrete FEM approximation and discuss its convergence rate. In Section \ref{sec_trunc}, we analyze the truncated KL expansion. In Section \ref{sec_qmc}, we describe the QMC quadrature rule for approximating the integral and derive an error bound. To complement the theoretical findings, we present several numerical experiments in Section \ref{sec_exp}. In Section \ref{sec_con}, we give concluding remarks. In the appendix, we collect several useful inequalities.

\section{Weak formulation of the parametric problem}\label{sec_pre}
We begin with the notations for useful function spaces. For $m\geq0$ and $1\leq  p\leq\infty$, we denote by $L^p(D)$ and $W^{m,p}(D)$ the standard Lebesgue and Sobolev spaces, and  when $p=2$, we write $H^m(D):=W^{m,2}(D)$. We denote by $H^1_0(D)$ the space of functions in $H^1(D)$ with zero trace on $\partial D$ and by $C^{\infty}_0(D)$ the space of $C^{\infty}$ functions with compact support in $D$. The notation $X(D)^d$ denotes the $d$-fold product space of $X(D)$, and $\boldsymbol{a}\cdot \boldsymbol{b}$ denotes the scalar product of two vectors $\boldsymbol{a}$ and $\boldsymbol{b}$. Also, $C$ denotes a generic positive constant, which may change at each appearance, and $A\lesssim B$ means that there exists $C>0$ such that $A\leq CB$.

Next, we define the following functional spaces that are frequently used in the study of incompressible fluid flow:
\begin{align*}
\mathcal{V}&=\{\uu\in C^{\infty}_0(D)^2: {\rm{div}}\,\uu=0\},\quad
V=\overline{\mathcal{V}}^{H^1_0(D)^2},\quad \mbox{and}\quad
H=\overline{\mathcal{V}}^{L^2(D)^2}.
\end{align*}
The space $H$ is a Hilbert space with the $L^2(D)$ inner product $(\cdot,\cdot)$, and the space $V$ is equipped with the scalar product
$a(\uu,\vv)=\int_D\nabla\uu:\nabla\vv\dx.$ Further,
we define a trilinear form  $B:V\times V\times V\to \mathbb{R}$ by
\begin{equation}\label{conv_term}
B[\uu,\vv,\ww]=\frac{1}{2}\int_D\big(((\uu\cdot\nabla)\vv)\cdot\ww-((\uu\cdot\nabla)\ww)\cdot\vv\big)\dx.
\end{equation}
This modified convective term conserves the skew symmetry of the discrete divergence, cf. \eqref{DDF} below.
Note that $B[\cdot,\cdot,\cdot]$ coincides with the trilinear form associated with the
convection term in \eqref{main_eq1} for pointwise divergence-free functions.

%\begin{proposition}\label{prop:B}
%[Property of the trilinear form $B[\cdot,\cdot,\cdot]$.]
%\begin{itemize}
%\item
 %$B[\uu,\vv,\ww]=-B[\uu,\ww,\vv]$ for all $\uu, \vv,\ww\in H^1_0(D)^2$. In specific, $B[\uu,\vv,\vv]=0$ for all $\uu,\vv\in H^1_0(D)^2$.
%\item
%\end{itemize}
%\end{proposition}
%Note that Proposition \ref{prop:B} implies
%\begin{align*}
%B[\uu,\uu,\vv]= \int_D\left((\uu\cdot\nabla)\uu\right)\cdot\vv\dx\text{ for all } \uu, \vv\in V.
%\end{align*}

We have the following well-posedness of problem \eqref{main_eq1}-\eqref{main_ic} in the 2D case \cite{Temam}.
\begin{theorem}\label{NS_existence}
For any given $\uu^0\in H$, there exists a unique $\uu\in L^{\infty}(0,T;H)\cap L^2(0,T;V)$ such that
\begin{align}\label{exist_1}
\left\{\begin{aligned}
(\partial_t\uu,\vv)+B[\uu,\uu,\vv]+a(\uu,\vv)&=0,\quad\forall\vv\in V,\\
\uu(0)&=\uu^0.
\end{aligned}\right.
\end{align}
Furthermore, the following energy estimate holds
\begin{equation}\label{ener_ineq}
\sup_{0<t<T}\|\uu(t)\|^2_{L^2(D)}+2\int^T_0\|\nabla\uu(s)\|^2_{L^2(D)}\,\mathrm{d}s\leq\|\uu^0\|^2_{L^2(D)}.
\end{equation}
\end{theorem}
It is direct to show that $\partial_t\uu\in L^2(0,T;V')$, which implies $\uu\in C([0,T];H)$. Thus the identity $\uu(0)=\uu^0 $ in \eqref{exist_1}
makes sense. We restrict our discussion to the 2D case for which the uniqueness of the weak solution to \eqref{exist_1} is known. The
analysis in this work can be extended to the 3D case {{if the solution is unique.}} This is known only under more restrictive assumptions; see e.g., \cite{Temam} for local well-posedness under the Serrin condition.

The weak formulation \eqref{exist_1} motivates the deterministic variational formulation of problem \eqref{main_eq1}-\eqref{main_ic}. {{Following the standard practice \cite{KL_1,KL_2},
the given Gaussian random field $Z(\cdot,\cdot)\in L^2(\Omega;L^2(D))$ in \eqref{lognormal_initial} can be parameterized by Karhunen-Lo\`{e}ve (KL) expansion:}}
\begin{equation}\label{KL}
Z(\xx,\omega)=\sum_{j=1}^{\infty}\sqrt{\mu_j}\xi_j(\xx)y_j(\omega), \quad (\xx,\omega)\in D\times\Omega,
\end{equation}
where $\{y_j\}_{j\geq1}$ is a sequence of independently and identically distributed (i.i.d.)
random variables following the standard normal distribution $\mathcal{N}(0,1)$. Let the
sequence $\{(\mu_j,\xi_j)\}_{j\geq1}$ be the eigenpairs of the covariance operator:
\begin{equation}\label{cov}
{\mathcal{C}}v(\xx)\defeq\int_D c(\xx,\xx')v(\xx')\,\mathrm{d}\xx'.
\end{equation}
Here, the kernel $c(\cdot,\cdot)$ denotes the covariance function of $Z(\cdot,\cdot)$ and is defined by
$c(\xx,\xx')=\mathbb{E}[Z(\xx,\cdot)Z(\xx',\cdot)]$ for any $ \xx,\xx'\in D$.
The integral operator $\mathcal{C}$ is self-adjoint and compact from $L^2(D)$ into $L^2(D)$. The non-negative eigenvalues, $\|c(\xx,\xx')\|_{L^2(D\times D)}\geq\mu_1\geq\mu_2\geq\cdots\geq0$, satisfy $\sum^{\infty}_{j=1}\mu_j=\int_D{\rm{Var}}(Z)(\xx)\,\mathrm{d}\xx$, and the eigenfunctions are orthonormal in $L^2(D)$, i.e. $\int_D\xi_j(\xx)\xi_k(\xx)\,\mathrm{d}\xx=\delta_{jk}$.

It follows from \eqref{lognormal_initial} and \eqref{KL} that the initial condition $\uu^0(\xx,\om)$ can be parametrized by an infinite-dimensional vector $\yy(\om)=(y_1(\om),y_2(\om),\cdots )\in\R^{\infty}$ (of i.i.d. Gaussian random variables $y_j\sim\mathcal{N}(0,1)$). The law of $\yy$ is defined on the probability space $(\R^{\infty},\mathcal{B}(\R^{\infty}),\bar{\boldsymbol{\mu}}_{G})$, where $\mathcal{B}(\R^{\infty})$ is the Borel $\sigma$-algebra on $\R^{\infty}$ and $\bar{\boldsymbol{\mu}}_{G}$ denotes the product Gaussian measure \cite{GM}, {\color{black}i.e.,} 
$\bar{\boldsymbol{\mu}}_G=\bigotimes^{\infty}_{j=1}\mathcal{N}(0,1).$

Throughout, we make the following regularity assumption on $Z(\xx,\omega)$.
\begin{assumption}\label{ass_initial}
For some $k>3$, $Z(\xx,\om)\in L^2(\Omega;H^k(D))$.
\end{assumption}

Under Assumption \ref{ass_initial}, the eigenvalues $\mu_j$ satisfy \cite{svd_decay}
\begin{equation}\label{eqn:eig-decay}
\mu_j\lesssim j^{-k-1} \mbox{ for $j\in\mathbb{N}$ sufficiently large}.
\end{equation}
Under Assumption \ref{ass_initial}, we have
$\|\xi_j\|_{H^{\theta k}(D)}\lesssim j^{\frac{\theta (k+1)}{2}}$ for any $0\leq\theta\leq1$ \cite{KL_GL}. This and
Sobolev embedding directly imply
\begin{equation}\label{eigen_est}
\|\xi_j\|_{C(D)}\lesssim j\quad{\rm{and}}\quad\|\nabla\xi_j\|_{C(D)}\lesssim j^{3/2} \text{ for } j\in\mathbb{N} \text{ sufficiently large}.
\end{equation}
We define the sequence $\boldsymbol{b}=\{b_j\}_{j\geq1}$ by
\begin{equation}\label{eqn:b}
b_j=\sqrt{\mu_j}\|\xi_j\|_{C(D)},\quad j\geq1.
\end{equation}
For the QMC analysis in Section \ref{sec_qmc}, we make the following assumption.
\begin{assumption}\label{ass_indiv}
There exists some $p\in (0,1]$ such that $\sum_{j\geq1}b_j^p<\infty$.
\end{assumption}

If $k>\frac{2}{p}+1$, then Assumption \ref{ass_initial} and \eqref{eqn:eig-decay} imply Assumption \ref{ass_indiv}.

Next, we define the following admissible parameter set
\[U_{\bb}\defeq\Big\{\yy\in\R^{\infty}:\sum_{j\geq1}b_j|y_j|<\infty\Big\}\subset\R^{\infty}.\]
Note that the set $U_{\bb}\subset\R^{\infty}$ is not a product of subsets of $\R$. However,
if Assumption \ref{ass_indiv} holds for some $0<p<1$, the set is $\bar{\boldsymbol{\mu}}_G$-measurable and of full Gaussian measure, i.e.,
$\bar{\boldsymbol{\mu}}_G(U_{\bb})=1$ \cite[Lemma 2.28]{admi_set}. Since $\bar{\boldsymbol{\mu}}_G(U_{\bb})=1$, we can use $U_{\bb}$ as the parameter space instead of $\R^{\infty}$. Note that $U_{\bb}$ is not a product domain. But we can define a product measure, e.g., $\bar{\boldsymbol{\mu}}_G$ on $U_b$ by restriction. Thus, we can identify the random initial condition $\uu^0(\xx,\om)$ with its parametric representation $\uu^0(\xx,\yy(\om))$: for each $\yy\in U_{\bb}$, we define the deterministic initial condition by
\begin{equation}\label{p_KL1}
\uu^0(\xx,\yy) = \left(-\partial_{x_2}\exp\bigg(\sum_{j=1}^{\infty}\sqrt{\mu_j}\xi_j(\xx)y_j\bigg),\partial_{x_1}\exp\bigg(\sum_{j=1}^{\infty}\sqrt{\mu_j}\xi_j(\xx)y_j\bigg)\right).
\end{equation}
Accordingly, we consider the following deterministic and parametric variational formulation of problem \eqref{main_eq1}--\eqref{main_ic}: For each $\yy\in U_{\bb}$, find $\uu(\yy)\in V$ satisfying
\begin{align}\label{eqn:WF}
\left\{\begin{aligned}
(\partial_t\uu(\yy),\vv)+B[\uu(\yy),\uu(\yy),\vv]+a(\uu(\yy),\vv)&=0,\quad\forall\vv\in V,\\
\uu(0,\yy)&=\uu^0(\yy).%\label{WF_2}
\end{aligned}\right.
\end{align}
Note that $\uu^0(\yy)\in H$ for each $\yy\in U_{\bb}$, due to Assumption \ref{ass_initial}.
This and Theorem \ref{NS_existence} imply that the solution $\uu(\yy)$ is uniquely determined
for each $\yy\in U_{\bb}$.

\section{Finite element approximation}\label{sec_fem}
We employ the implicit Euler method in time, and conforming FEM approximation in space for the variational formulation \eqref{eqn:WF} of the Navier--Stokes system. Let $\mathcal{T}_h$ be a shape-regular partition of the domain $D$ with a mesh size  $h$. We define conforming FEM spaces for the velocity $H_h\subset H^1_0(D)^2$ and the pressure $Q_h\subset L^2_0(D)$ by
\begin{align*}
H_h&=\{\mathbf{W}\in C(\overline{D})^2:\mathbf{W}|_K\in P_i(K)^2,\,\,\forall K\in\mathcal{T}_h\,\,{\rm{and}}\,\,\mathbf{W}|_{\partial D}=\boldsymbol{0}\},\\
L_h&=\{\Pi\in L^2_0(D):\Pi|_K\in P_j(K),\,\,\forall K\in\mathcal{T}_h\},
\end{align*}
where $P_i(K)$ denotes the space of polynomials of degree $i$ on the triangle $K\in\mathcal{T}_h$.
The FEM spaces $H_h$ and $L_h$ are assumed to satisfy the discrete inf-sup condition:
\begin{equation}\label{eqn:inf-sup}
\|q_h\|_{L^2(D)}\leq C \sup_{\vv_h\in H_h\setminus\{\boldsymbol{0}\}}\frac{(\nabla\cdot\vv_h,q_h)}{\|\nabla\vv_h\|_{L^2(D)}},\quad\forall q_h\in L_h,
\end{equation}
where the constant $C>0$ is independent of $h$. This condition holds for several finite element spaces,
e.g., Taylor-Hood element or MINI element \cite{girault}. Also we define a discrete divergence-free subspace of $H_h$ by
\begin{equation}\label{DDF}
V_h\defeq\{\vv_h\in H_h:(\nabla\cdot\vv_h,q_h)=0,\quad\forall q_h\in Q_h\}.
\end{equation}

Given $\ell\in\mathbb{N}$, we divide the interval $[0,T]$ into a uniform grid
$0=t_0<t_1<\cdots<t_{\ell}=T$, with $t_j=j\tau $, $j=0,1,\cdots,\ell$ and the time step size $\tau:=T/{\ell}$.
We employ the following implicit Euler conforming finite element approximation: find $\uu^{j+1}_{h}\in V_h$ for $j\in\{0,\cdots,\ell-1\}$ such that
\begin{align}
\left\{\begin{aligned}\label{SIBE}
\tau^{-1}(\uu^{j+1}_{h}-\uu^j_{h},\vv_h)+B[\uu^{j+1}_{h},\uu^{j+1}_{h},\vv_h]+a(\uu^{j+1}_{h},\vv_h)=0,&\quad\forall\vv_h\in V_h,
\\
(\uu_{h}^0,\vv_h)=(\uu^{0},\vv_h),&\quad\forall\vv_h\in V_h.
\end{aligned}\right.
\end{align}

The scheme \eqref{SIBE} is well-posed \cite{girault, Temam}, and satisfies the following stability result \cite{girault}.
\begin{proposition}\label{disc-well}
Let the time step size $\tau>0$ be sufficiently small. Then the discrete problem \eqref{SIBE} has a unique solution $\uu_h^{j}\in V_h$,
and furthermore the following stability estimate holds
\begin{equation}\label{stability}
\|\uu^{j+1}_h\|^2_{L^2(D)}+\tau\sum^j_{k=0}\|\nabla\uu^{k+1}_h\|^2_{L^2(D)}\leq \|\uu^0_{h}\|^2_{L^2(D)},\quad
\forall j=0,1,\cdots,\ell-1.
\end{equation}
\end{proposition}

Now we discuss the error estimate for the scheme \eqref{SIBE}, which has a
long and rich history (see, e.g., \cite{NS_est_1, girault,Temam, volker} for
a rather incomplete list).
%NS_est_3, NS_est_4, NS_est_5, NS_est_6, NS_est_7, NS_est_8, NS_est_9,  NS_est_10, NS_est_11, .
The scheme has a first-order convergence in time,
while the convergence order for spatial discretization depends on the regularity of solutions.
More precisely, we decompose the error $\uu(t_J)-\uu^J_h$ of the fully-discrete scheme \eqref{SIBE} into
\[\uu(t_J)-\uu^J_
h = (\uu(t_J)-\uu_h(t_J))+(\uu_h(t_J)-\uu^J_h),\]
where $\uu_h$ is the solution to the semi-discrete scheme
\begin{align*}
\left\{\begin{aligned}
(\partial_t\uu_h,\vv_h)+B[\uu_h,\uu_h,\vv_h]+a(\uu_h,\vv_h)&=0,\quad\forall\vv_h\in V_h,\\
(\uu^0_{h},\vv_h)&=(\uu_{0},\vv_h),\quad\forall\vv_h\in V_h.
\end{aligned}\right.
\end{align*}
The terms $\uu(t_J)-\uu_h(t_J)$ and $\uu_h(t_J)-\uu^J_h$ denote the spatial and
time discretization errors, respectively. For the spatial error, if
$\|\uu^0\|_{H^1(D)}\leq M$ for some $M>0$, then \cite[Theorem 4.5]{suli} gives
\begin{equation}\label{space_error}
\|\uu(t)-\uu_h(t)\|_{L^2(D)}\leq Ct^{-\frac{1}{2}}h^2, \quad \forall t\in (0,T],
\end{equation}
where the constant $C:=C(M,T,D)$ depends only on $M$, $T$ and $D$.
Note that the initial data $\uu^0$ also depends on the stochastic variable $\yy\in U_b$. To make the spatial discretization error independent of the stochastic variable $\yy\in U_b$,
we further assume that there exists some constant $M>0$ such that
\begin{align}\label{assumption:zzz}
\sup_{\yy\in U_{\bb}}\|\uu^0(\yy)\|_{H^1(D)}< M.
\end{align}

%Nonetheless, \eqref{lognormal_initial} and
%Assumption \ref{ass_initial} guarantee the existence of an $M>0$ independent of $\yy\in U_b$ such that
%$\sup_{\yy\in U_{\bb}}\|\uu^0(\yy)\|_{H^1(D)}< M$.
%Thus, the constant $M$ in \eqref{space_error} can be chosen to be independent of $\yy\in U_b$ for the solutions with the stochastic variable $\yy$.

{{For the temporal error}}, there are several results depending on the type of temporal discretization \cite{heywood2, heywood4, Temam, girault, volker}, under different time regularity assumptions on the solution. Since a detailed discussion of the finite element approximation of the Navier--Stokes equations is not our main focus, we will not explicitly specify the regularity assumptions. Instead, we assume the following condition: for any sufficiently small $h>0$,
\begin{equation}\label{time_ass}
\mathbb{E}\left[\|\partial_t\uu_h(\yy)\|_{L^2(0,T;V)}+\|\partial_{tt}\uu_h(\yy)\|_{L^2(0,T;V')}\right]\leq C,
\end{equation}
for some $C>0$ independent of $\yy\in U_{\bb}$. See, e.g., \cite{heywood4, reg_ref_1, reg_ref_2} for discussions on similar properties of $\uu_h$.
Then following the argument of \cite{Temam, girault}, we  obtain
\begin{equation}\label{time_error_est}
\mathbb{E}\left[\|\uu_h(t_J)-\uu^J_h\|_{L^2(D)}\right]\leq C\tau, \quad \forall J>0,
\end{equation}
where the constant $C>0$ is independent of $\yy\in U_{\bb}$.

Summarizing the above discussion and using the linearity of $\mathcal{G}$ and H\"older's inequality (for
the term related to $\mathcal{G}$), we have the following result.
\begin{theorem}\label{fem_main_thm}
Let $D\subset \mathbb{R}^2$ be a convex polygonal domain and $\mathcal{G}\in(L^2(D)^2)'$, and let \eqref{assumption:zzz} and \eqref{time_ass} hold. 
Then there exists a $C>0$ independent of $\yy\in U_{\bb}$ such that
\begin{equation*}
\mathbb{E}[|\mathcal{G}(\uu(t_J))-\mathcal{G}(\uu^J_h)|]\leq C(t_J^{-\frac{1}{2}}h^2+\tau).
\end{equation*}
\end{theorem}

\section{Truncation of the infinite-dimensional problem}\label{sec_trunc}

The second discretization step reduces the parametric dimension: we truncate the infinite series \eqref{KL} to a finite-dimensional summation, and obtain a truncated initial data
\begin{equation}\label{p_KL2}
\uu^0_{s}(\xx,\yy) =(-\partial_{x_2}\exp Z_s,\partial_{x_1}\exp Z_s),
\end{equation}
where $Z_s(\cdot,\cdot)$ is the truncated KL expansion of the random field $Z(\cdot,\cdot)$, given by
\[
Z_s(\xx,\yy):=\sum_{j=1}^{s}\sqrt{\mu_j}\xi_j(\xx)y_j.
\]
Note that $\uu^0_{s}(\xx,\yy)$ can be identified as the initial data $\uu^0(\xx,\yy)$ evaluated at the vector $\yy=(y_1,\cdots,y_s,0,0,\cdots)$. For any set of `active' coordinates $\kappa\subset\N$, we denote the vectors $\yy\in U_{\bb}$ with $y_j=0$ for $j\notin\kappa$ by $(\yy_{\kappa};\boldsymbol{0})$.

Then consider the following scheme for the truncated problem with the initial data $\uu^0_s$:
 Given $\uu^j_{s,h}\in V_h$, and $s\in\mathbb{N}$, find $\uu^{j+1}_{s,h}\in V_h$ such that
\begin{align}\label{trunc_SIBE}
\left\{\begin{aligned}
\tau^{-1}(\uu^{j+1}_{s,h}-\uu^j_{s,h},\vv_h)+B[\uu^{j+1}_{s,h},\uu^{j+1}_{s,h},\vv_h]+a(\uu^{j+1}_{s,h},\vv_h)=0,&\quad\forall\vv_h\in V_h,\\
(\uu^0_{s,h},\vv_h)=(\uu^0_{s},\vv_h),& \quad\forall\vv_h\in V_h,
\end{aligned}\right.
\end{align}
where $\uu^0_{s,h}$ is the $L^2(D)$-projection of $\uu^0_{s}$ into $V_h$.
The well-posedness of the scheme \eqref{trunc_SIBE} follows by the argument in Section \ref{sec_fem}.
The next result gives the truncation error.
\begin{theorem}\label{truncation_thm}
Let Assumption \ref{ass_initial} hold and $\mathcal{G}\in (L^2(D)^2)'$. If the time step size $\tau>0$ is sufficiently small, i.e.,
\begin{align}\label{eq:assumption-temporalstep}
\tau\|\uu^j_h\|^4_{L^4(D)}<1,\quad \forall j>0 \text{ and } h>0,
\end{align}
then for any time level $j\in\{0,1,\cdots,\ell\}$, mesh size $h>0$ and parametric dimension $s\in\mathbb{N}$,
there exists $C>0$ independent of $j,h$ and $s$ such that
\[
\mathbb{E}[|\mathcal{G}(\uu^j_h)-\mathcal{G}(\uu^j_{s,h})|]\leq C s^{-\frac{k}{2}+\frac{3}{2}}.
\]
\end{theorem}
\begin{proof}
First, by Fernique's theorem in Theorem \ref{Fernique}, for any $s\in\mathbb{N}$, we have
\begin{equation}\label{eqn:bdd-Fernique}
\|\exp(Z)\|_{L^2(\Omega;C(D))}\lesssim 1\quad{\rm{and}}\quad \|\exp(Z_s)\|_{L^2(\Omega;C(D))}\lesssim 1,
\end{equation}
(see, e.g., Proposition 5.6 in \cite{KL_GL}).
Let $\uu^j_h$ and $\uu^j_{s,h}$ be the solution to the scheme \eqref{SIBE} and problem \eqref{trunc_SIBE},
respectively, and $\ff^{i}_{s,h}:=\uu^{i}_h-\uu^{i}_{s,h}$.
%Since $\uu\in L^{\infty}(0,T;H)\cap L^2(0,T;V)$, by \eqref{2d_ineq} we have for each $\om\in\Omega$,
%\begin{equation}\label{L4_est}
%\int^T_0\|\uu(t,\yy)\|^4_4\dt
%\lesssim \int^T_0\|\uu(t,\yy)\|^2_2\|\nabla \uu(t,\yy)\|^2_2\dt \lesssim \|\uu(\yy)\|^2_{L^{\infty}(0,T;H)}\|\uu(\yy)\|^2_{L^2(0,T;V)}\lesssim \|\uu^0(\yy)\|_2^4.
%\end{equation}
%The last inequality above follows from Theorem \ref{NS_existence}.
Then subtracting \eqref{SIBE} from \eqref{trunc_SIBE} and setting $\vv_h=\ff^{j+1}_{s,h}\in V_h$ yield
\begin{equation*}%\label{4_1}
\tau^{-1}(\ff^{j+1}_{s,h}-\ff^{j}_{s,h},\ff^{j+1}_{s,h})+(\nabla\ff^{j+1}_{s,h},\nabla\ff^{j+1}_{s,h})=B[\uu^{j+1}_h,\uu^{j+1}_h,\ff^{j+1}_{s,h}] -B[\uu^{j+1}_{s,h},\uu^{j+1}_{s,h},\ff^{j+1}_{s,h}].
\end{equation*}
By the skew symmetry of $B[\cdot,\cdot,\cdot]$, H\"older's inequality,
Ladyzhenskaya's inequality in Lemma \ref{2d_est} and Young's inequality, we deduce
\begin{equation}\label{4_3}
\begin{aligned}
&\big|B\big[\uu^{j+1}_h,\uu^{j+1}_h,\ff^{j+1}_{s,h}]-B\big[\uu^{j+1}_{s,h},\uu^{j+1}_{s,h},\ff^{j+1}_{s,h}\big]\big|\\
 =& \big|B\big[\ff^{j+1}_{s,h},\uu^{j+1}_{h},\ff^{j+1}_{s,h}\big]+B\big[\uu^{j+1}_{s,h},\ff^{j+1}_{s,h},\ff^{j+1}_{s,h}\big]\big|\\
=&\big| -B\big[\ff^{j+1}_{s,h},\ff^{j+1}_{s,h},\uu^{j+1}_h\big]\big|
\lesssim\|\ff^{j+1}_{s,h}\|_{L^4(D)}\|\nabla \ff^{j+1}_{s,h}\|_{L^2(D)}\|\uu^{j+1}_{h}\|_{L^4(D)}\\
\lesssim&\|\ff^{j+1}_{s,h}\|^{1/2}_{L^2(D)}\|\nabla \ff^{j+1}_{s,h}\|^{3/2}_{L^2(D)}\|\uu^{j+1}_{h}\|_{L^4(D)}\\
\leq&\tfrac{c}{4}\|\uu^{j+1}_h\|^4_{L^4(D)}\|\ff^{j+1}_{s,h}\|^2_{L^2(D)}+\tfrac{3}{4}\|\nabla\ff^{j+1}_{s,h}\|^2_{L^2(D)}.
\end{aligned}
\end{equation}
Using the identity
\[\tau^{-1}(\ff^{j+1}_{s,h}-\ff^{j}_{s,h},\ff^{j+1}_{s,h})=(2\tau)^{-1}(\|\ff^{j+1}_{s,h}\|^2_{L^2(D)}
-\|\ff^{j}_{s,h}\|^2_{L^2(D)}+\|\ff^{j+1}_{s,h}-\ff^j_{s,h}\|^2_{L^2(D)}),\]
and summing up the previous results from $j=0$ to $j=k-1$ for $k\geq1$ yield
\begin{align*}%\label{4_4}
\|\ff^k_{s,h}\|^2_{L^2(D)}+\frac{\tau}{2}\sum^{k-1}_{j=0}\|\nabla \ff^{j+1}_{s,h}\|^2_{L^2(D)}
\lesssim &\|\ff^0_{s,h}\|^2_{L^2(D)}+\frac{\tau}{2}\sum^{k-1}_{j=0}\|\uu^{j+1}_h\|^4_{L^4(D)}\|\ff^{j+1}_{s,h}\|^2_{L^2(D)}.
\end{align*}
Then \eqref{eq:assumption-temporalstep} and the discrete Gronwall's inequality in Lemma \ref{gron} imply
%\begin{equation}\label{4_5}
\begin{align*}
\|\ff^k_{s,h}\|^2_{L^2(D)}
& \lesssim\|\ff^0_{s,h}\|^2_{L^2(D)}\exp\left(\tau\sum^{k-1}_{j=0}\frac{\|\uu^{j+1}_h\|^4_{L^4(D)}}{1-\tau\|\uu^{j+1}_h\|^4_{L^4(D)}}\right)\\
&\lesssim\|\ff^0_{s,h}\|^2_{L^2(D)}\exp\left(C\tau\sum^{k-1}_{j=0}\|\uu^{j+1}_h\|^4_{L^4(D)}\right).
\end{align*}
Note that the exponential term on the right-hand side is uniformly bounded in $k\geq1$. Indeed,
Proposition \ref{disc-well} and Ladyzhenskaya's inequality (cf. Lemma \ref{2d_est} in the appendix)
yield the following $L^4(D)$-estimate on $\uu_h^{j}$
\begin{equation}\label{eqn:uhJ-L4}
\tau\sum^{j-1}_{k=0}\|\uu^{k+1}_h\|^4_{L^4(D)}\leq 2\|\uu^0_{h}\|^4_{L^2(D)}.
\end{equation}
Further, by taking $\vv_h=\uu^0_{h}-\uu^0_{s,h}$ in the definition of $L^2$-projection, i.e.,
$(\uu^0-\uu^0_{s},\vv_h)=(\uu^0_{h}-\uu^0_{s,h},\vv_h)$, and H\"older's inequality, we deduce
\[\|\uu^0_{h}-\uu^0_{s,h}\|_{L^2(D)}\leq\|\uu^0-\uu^0_{s}\|_{L^2(D)}\lesssim \|\uu^0-\uu^0_{s}\|_{C(D)},\]
since the domain $D$ is bounded. To estimate $\|\uu^0-\uu^0_{s}\|_{C(D)}$, we derive
from the definitions of $\uu^0$ and $\uu^0_{s}$ that
\begin{align*}
&\mathbb{E}\left[\|\uu^0-\uu^0_{s}\|_{C(D)}\right]\\
=&\mathbb{E}\left[\|(-\partial_{x_2}Z\exp(Z)+\partial_{x_2}Z_s\exp(Z_s),\partial_{x_1}Z\exp(Z)-\partial_{x_1}Z_s\exp(Z_s))\|_{C(D)}\right]\\
\lesssim& \sum_{i=1}^2\mathbb{E}\left[\|\partial_{x_i}Z\exp(Z)-\partial_{x_i}Z_s\exp(Z_s)\|_{C(D)}\right].
\end{align*}
Then the triangle inequality, H\"older's inequality, Assumption \ref{ass_initial} and the Sobolev embedding yield
\begin{align*}
&\hspace{5mm}\mathbb{E}\left[\|\partial_{x_i}Z\exp(Z)-\partial_{x_i}Z_s\exp(Z_s)\|_{C(D)}\right]\\
&\lesssim \mathbb{E}\left[\|\partial_{x_i}Z\exp(Z)-\partial_{x_i}Z\exp(Z_s)\|_{C(D)}\right]+\mathbb{E}\left[\|\partial_{x_i}Z\exp(Z_s)-\partial_{x_i}Z_s\exp(Z_s)\|_{C(D)}\right]\\
&\lesssim \mathbb{E}\left[\|\partial_{x_i}Z\|_{C(D)}\|\exp(Z)-\exp(Z_s)\|_{C(D)}\right]+\mathbb{E}\left[\|\exp(Z_s)\|_{C(D)}\|\partial_{x_i}Z-\partial_{x_i}Z_s\|_{C(D)}\right]\\
&\lesssim\mathbb{E}\left[\|\exp(Z)-\exp(Z_s)\|_{C(D)}\right]+\mathbb{E}\left[\|\exp(Z_s)\|_{C(D)}^2\right]^{\frac{1}{2}}\times\mathbb{E}\left[\|\partial_{x_i}Z-\partial_{x_i}Z_s\|^2_{C(D)}\right]^{\frac{1}{2}}\\
&=:{\rm{I}}+{\rm{II}}\times{\rm{III}}.
\end{align*}
Note that under Assumption \ref{ass_initial}, for sufficiently large $s\in\mathbb{N}$, there holds \cite{KL_GL}
\begin{equation}\label{decay_estimate}
\|Z-Z_s\|_{L^2(\Omega;C(D))}\lesssim s^{{-\frac{k}{2}}+1}.
\end{equation}
Thus, for the term ${\rm{I}}$, it follows from the inequality $|e^x-e^y|\leq|x-y|(e^x+e^y)$ for
any $x,y\in\R$, H\"older's inequality, and \eqref{eqn:bdd-Fernique} that
\begin{align*}
{\rm{I}}
& \lesssim \mathbb{E}\left[\|Z-Z_s\|_{C(D)}\|\exp(Z)+\exp(Z_s)\|_{C(D)}\right]\\
& \lesssim \|Z-Z_s\|_{L^2(\Omega;C(D))}\|\exp(Z)+\exp(Z_s)\|_{L^2(\Omega;(C(D))}\\
& \lesssim \|Z-Z_s\|_{L^2(\Omega;C(D))}\left(\|\exp(Z)\|_{L^2(\Omega;C(D))}+\|\exp(Z_s)\|_{L^2(\Omega;C(D))}\right)
 \lesssim s^{-\frac{k}{2}+1}.
\end{align*}
Meanwhile, \eqref{eqn:bdd-Fernique} directly implies ${\rm{II}}\lesssim 1$. Since $\{y_j(\om)\}_{j\in\mathbb{N}}$ is
orthonormal in $L^2(\Omega)$, using the estimates \eqref{eqn:eig-decay} and \eqref{eigen_est} leads to
\begin{align*}
%\|\partial_{x_i}Z-\partial_{x_i}Z_s\|_{L^2(\Omega;C(D))}
{\rm III}&=\bigg\|\sum_{j> s}\sqrt{\mu_j}\partial_{x_i}\xi_j y_j\bigg\|_{L^2(\Omega;C(D))}\lesssim\bigg\|\sum_{j\geq s}\sqrt{\mu_j}\|\partial_{x_i}\xi_j\|_{C(D)}y_j\bigg\|_{L^2(\Omega)}\\
&\lesssim\bigg(\sum_{j> s}\mu_j\|\partial_{x_i}\xi_j\|^2_{C(D)}\bigg)^{\frac{1}{2}}\lesssim\bigg(\sum_{j> s}j^{-k-1}j^3\bigg)^{\frac{1}{2}}
\lesssim\left(\int^{\infty}_s t^{-k+2}\dt\right)^{\frac{1}{2}}
\lesssim s^{-\frac{k}{2}+\frac{3}{2}},
\end{align*}
upon noting $k>2$ in Assumption \ref{ass_initial}. Finally, we derive
\begin{align*}
\mathbb{E}[|\mathcal{G}(\uu^j_h)-\mathcal{G}(\uu^j_{s,h})|]&= \mathbb{E}[|\mathcal{G}(\uu^j_h-\uu^j_{s,h})|]\lesssim \mathbb{E}[\|\mathcal{G}\|_{(L^2(D))'}\|\uu^j_h-\uu^j_{s,h}\|_{L^2(D)}]\\
&\lesssim \mathbb{E}[\|\uu^0-\uu^0_{s}\|_{L^2(D)}] \lesssim s^{-\frac{k}{2}+\frac{3}{2}}.
\end{align*}
Combining the preceding estimates proves the desired assertion.
\end{proof}
\begin{remark}
The estimate \eqref{eqn:uhJ-L4} implies
$\tau\|\uu^j_h\|^4_{L^4(D)}\leq 2\|\uu^0\|^4_{L^2(D)}$ for all $j\geq0$ and $ h>0$.
Below we shall assume the smallness of initial data $\|\uu^0\|_{L^2(D)}\leq\varepsilon\approx\tau^{1/2}$. Thus,
for sufficiently small $\tau>0$, the condition \eqref{eq:assumption-temporalstep} holds independently of $j\geq 1$ and $h>0$.
\end{remark}

\section{Quasi-Monte Carlo integration}\label{sec_qmc}
Now we use QMC to approximate the linear functional $\mathcal{G}\in (L^2(D)^2)'$, i.e.,
$\mathbb{E}[\mathcal{G}(\uu^J_{s,h})]$. More precisely, given the parametric
dimension $s\in\N$, time level $J\in \{1,\cdots, \ell\}$ and mesh size $h>0$, we
approximate $\mathbb{E}[\mathcal{G}(\uu^J_{s,h})]$ by the following integral
\begin{equation}\label{QMC_term}
I_s(F^J_{s,h})\defeq\int_{\R^s}F^J_{s,h}(\yy)\prod^s_{j=1}\phi(y_j)\,{\textrm{d}}\yy, \quad{\rm{with}}\,\,\,F^J_{s,h}(\yy)=\mathcal{G}(\uu^J_{s,h}(\cdot,\yy)).
\end{equation}
Here, $\phi(y){\color{black}=}\exp(-y^2/2)/\sqrt{2\pi}$ and $\Phi(y)=\int^y_{-\infty}\exp(-t^2/2)/\sqrt{2\pi}\dt$
is the standard Gaussian probability density function and the cumulative normal distribution function, respectively. To apply QMC methods, we transform the integral \eqref{QMC_term}
over the unbounded domain $\R^s$ to a hyper-cube $[0,1]^s$ by introducing new variables $\yy=\boldsymbol{\Phi}_s^{-1}(\vv)$, where $\boldsymbol{\Phi}_s^{-1}(\vv)$ is the inverse cumulative normal distribution function for $\vv\in [0,1]^s$. Then changing variables yields
\[
\mathbb{E}\left[\mathcal{G}\left(\uu^J_{s,h}\right)\right]=\mathbb{E}\left[F^J_{s,h}\right] =\int_{(0,1)^s}F^J_{s,h}(\boldsymbol{\Phi}^{-1}_s(\vv))\,{\textrm{d}}\vv.
\]
We approximate the integral on $(0,1)^s$ by a {randomly-shifted lattice rule} (RSLR):
\begin{equation}\label{QMC_main}
\mathcal{Q}_{s,N}(F^J_{s,h};\boldsymbol{\Delta})\defeq\frac{1}{N}\sum^N_{i=1}F^J_{s,h} \left(\boldsymbol{\Phi}^{-1}_s\left({\rm{frac}}\left(\frac{i\boldsymbol{z}}{N}+{\boldsymbol{\Delta}}\right)\right)\right),
\end{equation}
where $\boldsymbol{z}\in\N^s$ is a deterministic generating vector and ${\boldsymbol{\Delta}}$ is
a uniformly distributed random shift over $[0,1]^s$. We shall establish an $\mathcal{O}(N^{-1})$ bound on the root-mean-square error (RMSE)
\begin{equation*}
\textrm{RMSE}_{\rm qmc}(N):=(\mathbb{E}^{\boldsymbol{\Delta}}[(\mathbb{E}[\mathcal{G}(\uu^J_{s,h})] -\mathcal{Q}_{s,N}(F^J_{s,h};\boldsymbol{\Delta}))^2])^{1/2},
\end{equation*}
where $\mathbb{E}^{\boldsymbol{\Delta}}$ denotes taking expectation in the random shift ${\boldsymbol{\Delta}}\in[0,1]^s$, with a constant independent of $s\in\N$, $J\in\mathbb{N}$ and $h>0$.

\subsection{Solution regularity in the stochastic variables}
First we bound mixed first derivatives of $\uu_{s,h}^{J}(\yy)$ with respect to the parametric variable $\yy\in U_{\bb}$. This represents the main technical challenge in the QMC analysis. Below $\fuu=(\fu_j)_{j\in\N}$ denotes the standard multi-index of non-negative integers with $|\fuu|=\sum_{j\geq1}\fu_j<\infty$, and  $\partial^{\fuu}\uu$ denotes the mixed derivative of $\uu$ with respect to all variables in the multi-index $\fuu$. We focus on the mixed first derivative $\partial^{\fuu}$, i.e., $\fu_j\in\{0,1\}$ for all $j\in \mathbb{N}$.
By Assumption \ref{ass_indiv}, there exists $N_0\in\mathbb{N}$ sufficiently large such that
$j>N_0$  implies $b_j\leq\frac{1}{2}$. Then we define a sequence of positive numbers $\{C_j\}_{j=1}^{\infty}$ by
\begin{equation}\label{eqn:cj}
C_j:=\left\{
\begin{aligned}
&\max\{2b_j,1\} &&\text{ if }j\in\{1,\cdots,N_0\}; \\
&1 &&\text{ if } j>N_0.
\end{aligned}
\right.
\end{equation}
Then the constant $C^*>0$, defined by
\begin{align}\label{eqn:cc}
C^*:=\prod_{j\geq1}C_j=\prod_{j=1}^{N_0}\max\{2b_j,1\},
\end{align}
is finite, i.e., $C^*<\infty$.

Next, we estimate the mixed first derivatives of the discrete truncated initial data $\uu^0_{s,h}$.
%Next we estimate $\left\|\partial^{\fuu}\uu^0_{s}(\yy)\right\|_2$.
\begin{lemma}
Let $\partial^{\fuu}$ be a mixed first derivative with respect to $\yy$ and let Assumption \ref{ass_indiv} hold. Then the following estimate holds
\begin{equation}\label{mid_init_est}
\|\partial^{\fuu}\uu^0_{s}\|_{L^2(D)}\leq
2^{-|\fuu|}(C^*+2\sqrt{2})|\fuu|
\|\exp Z_s(\cdot,\yy)\|_{H^1(D)}.
\end{equation}
\end{lemma}
\begin{proof}
By the definition of $L^2(D)$-projection $\uu^0_{s,h}(\yy)$, for each $\yy \in U_{\bb}$, we have
\[(\uu^0_{s,h}(\yy),\vv_h)=(\uu^0_{s}(\yy),\vv_h),\quad\forall\vv_h\in V_h.\]
Taking $\partial^{\fuu}$ on both sides
%\[(\partial^{\fuu}\uu^0_{s,h}(\yy),\vv_h)=(\partial^{\fuu}\uu^0_{s}(\yy),\vv_h),\quad\forall\vv_h\in V_h,\]
and then setting $\vv_h=\partial^{\fuu}\uu^0_{s,h}\in V_h$ yield that for each $\yy\in U_{\bb}$,
\[\left\|\partial^{\fuu}\uu^0_{s,h}(\yy)\right\|^2_{L^2(D)}=\left(\partial^{\fuu}\uu^0_{s}(\yy),\partial^{\fuu}\uu^0_{s,h}(\yy)\right)
\leq\left\|\partial^{\fuu}\uu^0_{s}(\yy)\right\|_{L^2(D)}\left\|\partial^{\fuu}\uu^0_{s,h}(\yy)\right\|_{L^2(D)}.\]
Hence, we have
\begin{equation}\label{deriv_initial}
  \|\partial^{\fuu}\uu^0_{s,h}(\yy)\|_{L^2(D)}\leq\|\partial^{\fuu}\uu^0_{s}(\yy)\|_{L^2(D)}.
\end{equation}
With the notation $j\in\fuu$ when $\fuu_j=1$, by the product rule,
\begin{align*}
&\|\partial^{\fuu}\uu^0_{s}\|_{L^2(D)}
 =\|\partial^{\fuu}\partial_x\exp Z_s\|_{L^2(D)}=\bigg\|\partial_x\bigg[\bigg(\prod_{j\in\fuu}\sqrt{\mu_j}\xi_j\bigg)\exp Z_s\bigg]\bigg\|_{L^2(D)}\\
\leq&\bigg\|\bigg(\prod_{j\in\fuu}\sqrt{\mu_j}\xi_j\bigg)\partial_x\exp Z_s\bigg\|_{L^2(D)}\!\!\!+\bigg\|\sum_{i\in\fuu}\sqrt{\mu_i}
\partial_x\xi_i\bigg(\prod_{j\in\fuu,\,j\neq i}\sqrt{\mu_j}\xi_j\bigg)\exp Z_s\bigg\|_{L^2(D)}\!\!\!\\
=:&{\rm{I}}+{\rm{II}}.%\label{5_4}
\end{align*}
By \eqref{eqn:cj} and \eqref{eqn:cc}, the following inequality holds
\[
\prod_{j\in\fuu}b_j=\prod_{j\in\fuu,\,j>N_0}b_j\times \prod_{j\in\fuu,\,j\leq N_0}b_j  \leq \prod_{j\in\fuu,\,j>N_0}\frac{1}{2}\times \prod_{j\in\fuu,\,j\leq N_0}\frac{C_j}{2} \leq\frac{C^*}{2^{|\fuu|}}.
\]
Then we can bound the first term $\rm{I}$ by
\begin{equation*}%\label{5_1}
{\rm{I}}\leq\bigg(\prod_{j\in\fuu}b_j\bigg)\|\partial_x\exp Z_s\|_{L^2(D)}\leq \frac{C^*}{2^{|\fuu|}}\|\partial_x\exp Z_s\|_{L^2(D)}.
\end{equation*}
For the term ${\rm{II}}$, we discuss the cases $|\fuu|\geq2$ and $|\fuu|=1$ separately.
When $|\fuu|\geq2$, the estimates \eqref{eqn:eig-decay} and \eqref{eigen_est} imply
\begin{equation*}
\begin{aligned}
{\rm{II}}
&\leq \|\exp Z_s\|_{L^2(D)}\sum_{i\in\fuu}\sqrt{\mu_i}\|\partial_x\xi_i\|_{\infty}\bigg(\prod_{j\in\fuu,\,j\neq i}b_j\bigg)\\
&\leq \|\exp Z_s\|_{L^2(D)}\sum_{i\in\fuu}i^{-\frac{k}{2}+1}\bigg(\prod_{j\in\fuu,\,j\neq i}j^{-\frac{k}{2}+\frac{1}{2}}\bigg).
\end{aligned}
\end{equation*}
Plugging the following identity
\begin{align*}
\sum_{i\in\fuu}i^{-\frac{k}{2}+1}\bigg(\prod_{j\in\fuu,\,j\neq i}j^{-\frac{k}{2}+\frac{1}{2}}\bigg)
&=\sum_{i\in\fuu}i^{-\frac{k}{2}+1}\bigg(\prod_{j\in\fuu,\,j\neq i}j^{-\frac{k}{2}+1-\frac{1}{2}}\bigg)\\
&=\sum_{i\in\fuu}\bigg(\prod_{j\in\fuu}j^{-\frac{k}{2}+1}\bigg)\bigg(\prod_{j\in\fuu,\,j\neq i}j^{-\frac{1}{2}}\bigg)
\end{align*}
into the last inequality leads to
\begin{align*}
{\rm{II}}
& \leq\|\exp Z_s\|_{L^2(D)}\sum_{i\in\fuu}\bigg(\prod_{j\in\fuu}j^{-\frac{k}{2}+1}\bigg)\bigg(\prod_{j\in\fuu,\,j\neq i}j^{-\frac{1}{2}}\bigg)\\
&=\|\exp Z_s\|_{L^2(D)}\bigg(\prod_{j\in\fuu}j^{-\frac{k-2}{2}}\bigg)\sum_{i\in\fuu}\bigg(\prod_{j\in\fuu,\,j\neq i}j^{-\frac{1}{2}}\bigg).
\end{align*}
Upon noting $k>3$, we deduce
\begin{align*}
{\rm{II}}
&\leq\|\exp Z_s\|_{L^2(D)}\frac{1}{\left(2^{\frac{k-2}{2}}\right)^{|\fuu|-1}}\sum_{i\in\fuu}\bigg(\prod_{j\in\fuu,\,j\neq i}j^{-\frac{1}{2}}\bigg)\nonumber\\
&\leq\|\exp Z_s\|_{L^2(D)}\frac{1}{(\sqrt{2})^{|\fuu|-1}}\sum_{i\in\fuu}\frac{1}{(\sqrt{2})^{|\fuu|-2}}=\|\exp Z_s\|_{L^2(D)}\frac{2\sqrt{2}|\fuu|}{2^{|\fuu|}}.%\label{5_2}
\end{align*}
When $|\fuu|=1$, from the estimates \eqref{eqn:eig-decay} and \eqref{eigen_est}, we derive for some $r\in\mathbb{N}$ that
\begin{align*}%\label{5_3}
{\rm{II}}
&=\|\partial_x(\sqrt{\mu_r}\xi_r)\exp Z_s\|_{L^2(D)}\leq \|\exp Z_s\|_{L^2(D)}\sqrt{\mu_r}\|\partial_x\xi_r\|_{L^\infty(D)}\\
&\leq\|\exp Z_s\|_{L^2(D)}r^{-\frac{k-2}{2}}\leq 2\sqrt{2}|\fuu|2^{-|\fuu|}\|\exp Z_s\|_{L^2(D)}.
\end{align*}
Combining the preceding estimates completes the proof of the lemma.
\end{proof}

The following discussion requires a smallness of the initial log-normal random field $\exp Z(\xx,\yy)$
for each realization $\yy\in U_{\bb}$: there exists some $\varepsilon>0$ such that
\begin{align}\label{eqn:smallness}
\sup_{\yy\in U_{\bb}}\|\exp Z(\cdot,\yy)\|_{H^1(D)}<\varepsilon.
\end{align}
This condition on $\uu^0$ can be found in the theory of incompressible fluid flow problems;
see, e.g., \cite{Malek, Malek_NS}, where the existence of global strong solutions was discussed
under a smallness condition on $\uu^0$. Note that $\exp Z_s$ and $\nabla\exp Z_s=\nabla Z_s\exp Z_s$
coincide with the evaluations of $\exp Z$ and $\nabla\exp Z=\nabla Z\exp Z$ at $\yy_s=(y_1,\cdots,
y_s,0,0,\cdots)$, which belongs to $U_{\bb}$ for any $(y_1,\cdots,y_s)\in \R^s$. Hence,
Assumption \eqref{eqn:smallness} implies
\begin{equation}\label{trunc_small}
\sup_{\yy_s\in\R^s}\|\exp Z_s(\cdot,\yy_s)\|_{H^1(D)}<\varepsilon,\quad \forall s\in \mathbb{N}.
\end{equation}

{
\begin{remark}
The small data assumption \eqref{trunc_small} plays a vital role in the analysis. It is frequently used in the study of incompressible fluid flow problems both in theory and numerics \cite{Malek,Malek_NS}, including the stationary Navier-Stokes equations. Recently, in a similar context, it was used by Cohen et al \cite{holo_2018} to establish the holomorphic dependence of the velocity and the pressure of stationary Navier-Stokes equations on the domain, with the parameterized families of domain mappings over the uniform distribution; see \cite[equation (2.9)]{holo_2018} and the references therein. The problem under consideration is non-stationary and involves nonlinear random initial conditions, and thus it is expected to be even more challenging.

Next we briefly comment on the validity of the assumption. Under the definition of the admissible parameter set $U_b$, each realization of the initial data is bounded, and thus
\begin{align*}
 \|\exp Z(\cdot,\mathrm{y})\|_{H^1(D)}<\infty,\quad \forall \mathrm{y}\in U_b.
\end{align*}
Thus, it is reasonable to make the smallness assumption by working with a smaller admissible parameter set $V_{\boldsymbol{b}}=\{\boldsymbol{y}\in U_{\boldsymbol{b}}:\|\exp Z(\cdot,\mathrm{y})\|_{H^1(D)}<\varepsilon\}$ of $U_{\boldsymbol{b}}$. Even though the set  $V_{\boldsymbol{b}}$ is not of full measure, by appropriately adjusting the mean and covariance of the Gaussian random field $Z(\cdot,\cdot)$, one can make the probability of extracting $\boldsymbol{y}$ that satisfies the smallness assumption \eqref{trunc_small} sufficiently high. Thus, we can make the measure of $V_{\boldsymbol{b}}$ close to 1 for the specific choice of $Z(\cdot,\cdot)$. The subsequent analysis assumes the smallness condition \eqref{trunc_small}, and we leave a full investigation of the condition and its relaxation to future works.
\end{remark}}

Now we can state an \textit{a priori} estimate for $\|\partial^{\fuu}\uu_{s,h}^J(\yy)\|_{L^2(D)}$ for each $\yy\in U_{\bb}$.
\begin{theorem}\label{reg_estimate}
Let $\partial^{\fuu}$ be a mixed first derivative with respect to $\yy$ and let Assumption \ref{ass_indiv} hold. Then there exists $\varepsilon:=\varepsilon(\tau)>0$, proportional to $\tau^{\frac{1}{2}}$, such that if the smallness assumption \eqref{eqn:smallness} holds, then for any truncation dimension $s\in\mathbb{N}$, time level $J=1,2,\cdots,\ell$ and mesh size $h>0$, there holds with $C^*>0$ defined in \eqref{eqn:cc},
\[\|\partial^{\fuu}\uu_{s,h}^J(\yy)\|_{L^2(D)}\leq (2\sqrt{2}+2C^*)|\fuu|!2^{|\fuu|}\|\exp Z_s(\yy)\|_{H^1(D)},\quad\forall \yy\in U_{\bb}.\]
\end{theorem}
\begin{proof}
In the proof, we suppress $\yy$ from $\uu^j_{s,h}(\yy)$ etc. since it is fixed. When $|\fuu|=0$, the desired inequality
can be verified directly. Hence we discuss only the case $|\fuu|\geq1$.
By taking $\partial^{\fuu}$ in the discrete scheme
\eqref{trunc_SIBE} and the product rule, we obtain
\begin{align*}
\tau^{-1} (\partial^{\fuu}\uu^{j+1}_{s,h},&\vv_h)-\tau^{-1}(\partial^{\fuu}\uu^j_{s,h},\vv_h)+(\nabla\partial^{\fuu}\uu^{j+1}_{s,h},\nabla\vv_h)\\
&+\sum_{\fmm\preceq\fuu}{\fuu\choose\fmm}B[\partial^{\fuu-\fmm}\uu^{j+1}_{s,h},\partial^{\fmm}\uu^{j+1}_{s,h},\vv_h]=0,\quad \forall\vv_h\in V_h,
\end{align*}
where $\fmm\preceq\fuu$ denotes $\fm_j\leq\fu_j$ for all $j\geq1$, and
$\fuu-\fmm$ is a multi-index $(\fu_j-\fm_j)_{j\geq1}$ and ${\fuu\choose\fmm}\defeq\Pi_{j\geq1}{\fu_j\choose \fm_j}$.
Upon splitting the summation with $\fmm\neq \fuu$ (i.e., $\precneqq$), we obtain
\begin{align*}
(\nabla\partial^{\fuu}\uu^{j+1}_{s,h},&\nabla\vv_h)+B[\uu^{j+1}_{s,h},\partial^{\fuu}\uu^{j+1}_{s,h},\vv_h]+ \tau ^{-1}(\partial^{\fuu}\uu^{j+1}_{s,h},\vv_h)\\
  &-\tau^{-1}(\partial^{\fuu}\uu^j_{s,h},\vv_h)=-\sum_{\substack{\fmm\precneqq\fuu}}{\fuu\choose\fmm}B[\partial^{\fuu-\fmm}\uu^{j+1}_{s,h},\partial^{\fmm}\uu^{j+1}_{s,h},\vv_h],\quad \forall \vv_h\in V_h.
\end{align*}
Taking $\vv_h{\color{black}=}2\tau\partial^{\fuu}\uu^{j+1}_{s,h}$ in the identity yields
\begin{align}
&2\tau\|\nabla\partial^{\fuu}\uu^{j+1}_{s,h}\|^2_{L^2(D)}+\|\partial^{\fuu}\uu^{j+1}_{s,h}\|^2_{L^2(D)} -\|\partial^{\fuu}\uu^j_{s,h}\|^2_{L^2(D)} +\|\partial^{\fuu}\uu^{j+1}_{s,h}-\partial^{\fuu}\uu^j_{s,h}\|^2_{L^2(D)}
\nonumber\\
=&-2\tau\sum_{\fmm\precneqq\fuu}{\fuu\choose\fmm}B\left[\partial^{\fuu-\fmm}\uu^{j+1}_{s,h},\partial^{\fmm}\uu^{j+1}_{s,h}, \partial^{\fuu}\uu^{j+1}_{s,h}\right].\label{eqn:11111}
\end{align}
The skew symmetry of the trilinear form $B[\cdot,\cdot,\cdot]$ implies
\begin{align*}
&\big|B[\partial^{\fuu-\fmm}\uu^{j+1}_{s,h}, \partial^{\fmm}\uu^{j+1}_{s,h},\partial^{\fuu}\uu^{j+1}_{s,h}]\big|\\
\leq&\tfrac{1}{2}
\|\partial^{\fuu-\fmm}\uu^{j+1}_{s,h}\|_{L^4(D)} \|\nabla\partial^{\fmm}\uu^{j+1}_{s,h}\|_{L^2(D)} \|\partial^{\fuu}\uu^{j+1}_{s,h}\|_{L^4(D)} \\
&+ \tfrac{1}{2}\|\partial^{\fuu-\fmm}\uu^{j+1}_{s,h}\|_{L^4(D)}\|\nabla\partial^{\fuu}\uu^{j+1}_{s,h}\|_{L^2(D)} \|\partial^{\fmm}\uu^{j+1}_{s,h}\|_{L^4(D)}.
\end{align*}
By bounding the terms $\|\partial^{\fuu-\fmm}\uu^{j+1}_{s,h}\|_{L^4(D)}$,
$\|\partial^{\fuu}\uu^{j+1}_{s,h}\|_{L^4(D)}$ and $\|\partial^{\fmm}\uu^{j+1}_{s,h}\|_{L^4(D)}$
using Ladyzhenskaya's inequality in Lemma \ref{2d_est}
%e.g.,
%\begin{align*}
%  \|\partial^{\fuu-\fmm}\uu^{j+1}_{s,h}(\yy)\|_{L^4(D)} &\leq \sqrt{2}\|\partial^{\fuu-\fmm}\uu^{j+1}_{s,h}(\yy)\|^{\frac{1}{2}}_{L^2(D)}
%\|\nabla\partial^{\fuu-\fmm}\uu^{j+1}_{s,h}(\yy)\|^{\frac{1}{2}}_{L^2(D)},\\
% \|\partial^{\fuu}\uu^{j+1}_{s,h}(\yy)\|_{L^4(D)} & \leq \|\partial^{\fuu}\uu^{j+1}_{s,h}(\yy)\|^{\frac{1}{2}}_{L^2(D)}\|\nabla\partial^{\fuu}\uu^{j+1}_{s,h}(\yy)\|^{\frac{1}{2}}_{L^2(D)},\\
% \|\partial^{\fmm}\uu^{j+1}_{s,h}(\yy)\|_{L^4(D)} & \leq \|\partial^{\fmm}\uu^{j+1}_{s,h}(\yy)\|_{L^2(D)}^{\frac{1}{2}} \|\nabla\partial^{\fmm}\uu^{j+1}_{s,h}(\yy)\|_{L^2(D)}^{\frac{1}{2}},
%\end{align*}
and then applying Poincar\'e's inequality (with Poincar\'{e} constant $C_p$), we deduce
\begin{align*}
&\Big|-\sum_{\fmm\precneqq\fuu}{\fuu\choose\fmm}B\left[\partial^{\fuu-\fmm}\uu^{j+1}_{s,h}, \partial^{\fmm}\uu^{j+1}_{s,h},\partial^{\fuu}\uu^{j+1}_{s,h}\right]\Big|\\
\leq &\sqrt{2}C_p\sum_{\fmm\precneqq\fuu}{\fuu\choose\fmm}  \|\nabla\partial^{\fuu-\fmm}\uu^{j+1}_{s,h}\|_{L^2(D)} \|\nabla\partial^{\fmm}\uu^{j+1}_{s,h}\|_{L^2(D)}   \|\nabla\partial^{\fuu}\uu^{j+1}_{s,h}\|_{L^2(D)}.
\end{align*}
%\begin{align*}
%&\Big|-\sum_{\substack{\fmm\preceq\fuu \\ \fmm\neq\fuu}}{\fuu\choose\fmm}B\left[\partial^{\fuu-\fmm}\uu^{j+1}_{s,h}(\yy), \partial^{\fmm}\uu^{j+1}_{s,h}(\yy),\partial^{\fuu}\uu^{j+1}_{s,h}(\yy)\right]\Big|\\
%&\leq \frac{1}{2}\sum_{\substack{\fmm\preceq\fuu \\ \fmm\neq\fuu}}{\fuu\choose\fmm}
%\|\partial^{\fuu-\fmm}\uu^{j+1}_{s,h}(\yy)\|_{L^4(D)}\left(\|\nabla\partial^{\fmm}\uu^{j+1}_{s,h}(\yy)\|_{L^2(D)}\|\partial^{\fuu}\uu^{j+1}_{s,h}(\yy)\|_{L^4(D)}+ \|\nabla\partial^{\fuu}\uu^{j+1}_{s,h}(\yy)\|_{L^2(D)}\|\partial^{\fmm}\uu^{j+1}_{s,h}(\yy)\|_{L^4(D)}\right)\\
%&\leq \frac{\sqrt{2}}{2}\sum_{\substack{\fmm\preceq\fuu \\ \fmm\neq\fuu}}
%{\fuu\choose\fmm}\|\partial^{\fuu-\fmm}\uu^{j+1}_{s,h}(\yy)\|^{\frac{1}{2}}_{L^2(D)}
%\|\nabla\partial^{\fuu-\fmm}\uu^{j+1}_{s,h}(\yy)\|^{\frac{1}{2}}_{L^2(D)}
%\left(\|\nabla\partial^{\fmm}\uu^{j+1}_{s,h}(\yy)\|_{L^2(D)}\|\partial^{\fuu}\uu^{j+1}_{s,h}(\yy)\|^{\frac{1}{2}}_{L^2(D)}\|\nabla\partial^{\fuu}\uu^{j+1}_{s,h}(\yy)\|^{\frac{1}{2}}_{L^2(D)}\right.\\
%&\hspace{7cm}+\left.
%\|\nabla\partial^{\fuu}\uu^{j+1}_{s,h}(\yy)\|_{L^2(D)}\|\partial^{\fmm}\uu^{j+1}_{s,h}(\yy)\|_{L^2(D)}^{\frac{1}{2}}\|\nabla\partial^{\fmm}\uu^{j+1}_{s,h}(\yy)\|_{L^2(D)}^{\frac{1}{2}}\right)\\
%&\leq \sqrt{2}C_p\sum_{\substack{\fmm\preceq\fuu \\ \fmm\neq\fuu}}{\fuu\choose\fmm}  \|\nabla\partial^{\fuu-\fmm}\uu^{j+1}_{s,h}(\yy)\|_{L^2(D)} \|\nabla\partial^{\fmm}\uu^{j+1}_{s,h}(\yy)\|_{L^2(D)}   \|\nabla\partial^{\fuu}\uu^{j+1}_{s,h}(\yy)\|_{L^2(D)},
%\end{align*}
Together with \eqref{eqn:11111}, we arrive at
\begin{align*}
&2\tau\|\nabla\partial^{\fuu}\uu^{j+1}_{s,h}\|^2_{L^2(D)}+\|\partial^{\fuu}\uu^{j+1}_{s,h}\|^2_{L^2(D)} -\|\partial^{\fuu}\uu^j_{s,h}\|^2_{L^2(D)}+\|\partial^{\fuu}\uu^{j+1}_{s,h}-\partial^{\fuu}\uu^j_{s,h}\|^2_{L^2(D)}\\
&\leq 2\sqrt{2}C_p\tau\sum_{\fmm\precneqq\fuu}{\fuu\choose\fmm}  \|\nabla\partial^{\fuu-\fmm}\uu^{j+1}_{s,h}\|_{L^2(D)}  \|\nabla\partial^{\fmm}\uu^{j+1}_{s,h}\|_{L^2(D)}  \|\nabla\partial^{\fuu}\uu^{j+1}_{s,h}\|_{L^2(D)}.
\end{align*}
For fixed $J=2,\cdots,\ell$, let
$S^J_{\fmm}{\color{black}=}(\Delta t\sum^{J-1}_{j=0}\|\nabla\partial^{\fmm}\uu^{j+1}_{s,h}\|^2_{L^2(D)})^{\frac{1}{2}}.$
Summing up from $j=0$ to $j=J-1$, using
Cauchy--Schwarz inequality and \eqref{deriv_initial} yield
\begin{align*}
\|\partial^{\fuu}&\uu^J_{s,h}\|^2_{L^2(D)}+2\left(S^J_{\fuu}\right)^2 \leq\|\partial^{\fuu}\uu^0_{s,h}\|^2_{L^2(D)}\\
&+2\sqrt{2}C_p\tau\sum_{\fmm\precneqq\fuu }{\fuu\choose\fmm}\sum^{J-1}_{j=0}\|\nabla\partial^{\fuu-\fmm}\uu^{j+1}_{s,h}\|_{L^2(D)}\|\nabla\partial^{\fmm}\uu^{j+1}_{s,h}\|_{L^2(D)}   \|\nabla\partial^{\fuu}\uu^{j+1}_{s,h}\|_{L^2(D)}\\
&\leq\frac{2\sqrt{2}C_p}{\tau^{\frac{1}{2}}} \sum_{\fmm\precneqq\fuu}{\fuu\choose\fmm}S^J_{\fuu-\fmm}S^J_{\fmm}S^J_{\fuu} + \|\partial^{\fuu}\uu^0_{s,h}\|^2_{L^2(D)}.
\end{align*}
%[{\color{black}the last inequality contains an error? Cauchy-Schwarz gives $\ell^3$ norm in time?}]
Note that ${\fuu\choose\fmm}=1$, since $\fuu$ is a mixed first derivative, i.e. $\mathfrak{u}_j\in\{0,1\}$.
Then with $A:=2\sqrt{2}C_p\tau^{-\frac{1}{2}}$, we have
\begin{equation}\label{mid_mid}
\|\partial^{\fuu}\uu^J_{s,h}\|^2_{L^2(D)}+2(S^J_{\fuu})^2\leq  A \sum_{\fmm\precneqq\fuu}S^J_{\fuu-\fmm}S^J_{\fmm}S^J_{\fuu}+\|\partial^{\fuu}\uu^0_{s}\|^2_{L^2(D)}.
\end{equation}
To estimate $S^J_{\fuu}$ when $|\fuu|=0 $, by taking $\vv_h=2\tau\uu^{j+1}_{s,h}$ in \eqref{trunc_SIBE}
and Young's inequality,
\begin{align*}
2\|\uu^{j+1}_{s,h}\|^2_{L^2(D)}+2\tau\|\nabla\uu^{j+1}_{s,h}\|^2_{L^2(D)}
=(\uu^j_{s,h},\uu^{j+1}_{s,h})\leq \|\uu^j_{s,h}\|^2_{L^2(D)}+\|\uu^{j+1}_{s,h}\|^2_{L^2(D)},
\end{align*}
which implies
$$\|\uu^{j+1}_{s,h}\|^2_{L^2(D)}+2\tau\|\nabla\uu^{j+1}_{s,h}\|^2_{L^2(D)}\leq \|\uu^j_{s,h}\|^2_{L^2(D)}.$$
Summing the above inequality from $j=0$ to $j=J-1$ yields
\begin{align}
2(S^J_{\boldsymbol{0}})^2&\leq\|\uu^{J}_{s,h}\|^2_{L^2(D)}+2\tau\sum^{J-1}_{j=0}\|\nabla\uu^{j+1}_{s,h}\|^2_{L^2(D)}\nonumber\\
 &\leq \|\uu^0_{s,h}\|^2_{L^2(D)}\leq \|\uu^0_{s}\|_{L^2(D)}^2\leq\|\exp Z_s\|_{H^1(D)}.\label{case_1}
\end{align}
Now we choose $\varepsilon:=\frac{1}{4A(C^*+2\sqrt{2})}$,
which is proportional to $\tau^{1/2}$, and assume that the smallness condition \eqref{eqn:smallness} holds.
Then we shall prove the claim
\begin{equation}\label{desire_mid}
S^J_{\fuu}\leq |\fuu|!2^{|\fuu|}\|\partial^{\fuu}\uu^0_{s}\|_{L^2(D)}\quad\forall J\in\mathbb{N},
\end{equation}
by mathematical induction on $|\fuu|$.
For $|\fuu|=1$, from \eqref{mid_mid} and \eqref{case_1}, we have
\[
2(S^J_{\fuu})^2\leq A S^J_{\boldsymbol{0}}(S^J_{\fuu})^2+\|\partial^{\fuu}\uu^0_{s}\|^2_{L^2(D)}\leq A\|\exp Z_s\|_{H^1(D)}(S^J_{\fuu})^2+\|\partial^{\fuu}\uu^0_{s}\|^2_{L^2(D)}.
\]
This, the smallness assumption \eqref{eqn:smallness} and the choice of $\varepsilon$ imply the inequality \eqref{desire_mid} when $|\fuu|=1$.
Now suppose that \eqref{desire_mid} holds for all $\fuu$ with $1\leq|\fuu|\leq n-1$, $n\geq2$. For any given multi-index $\fuu$ with $|\fuu|=n$, it follows from \eqref{mid_mid} and \eqref{case_1} that
\[
2(S^J_{\fuu})^2\leq A\sum_{\fmm\precneqq\fuu,\,\fmm\neq\boldsymbol{0}}S^J_{\fuu-\fmm}S^J_{\fmm}S^J_{\fuu}+\frac{A}{2}\|\exp Z_s\|_{H^1(D)}(S^J_{\fuu})^2+\|\partial^{\fuu}\uu^0_{s}\|^2_{L^2(D)}.
\]
By the induction hypothesis, \eqref{mid_init_est} and the smallness assumption \eqref{eqn:smallness}, we have
\begin{align*}
\frac{(S^J_{\fuu})^2}{2}
&\leq A\sum^{|\fuu|-1}_{i=1}\sum_{|\fmm|=i}\|\partial^{\fuu}\uu^0_{s}\|^2_{L^2(D)} |\fuu-\fmm|!2^{|\fuu-\fmm|}|\fmm|!2^{|\fmm|}S^J_{\fuu}+\|\partial^{\fuu}\uu^0_{s}\|^2_{L^2(D)}\\
&\leq A\|\partial^{\fuu}\uu^0_{s}\|^2_{L^2(D)}(|\fuu|-1)!2^{|\fuu|}\sum^{|\fuu|-1}_{i=1}{|\fuu|\choose i}S^J_{\fuu}+\|\partial^{\fuu}\uu^0_{s}\|^2_{L^2(D)}\\
&\leq A\|\partial^{\fuu}\uu^0_{s}\|^2_{L^2(D)}(|\fuu|-1)!2^{|\fuu|}(2^{|\fuu|}-2)S^J_{\fuu}+\|\partial^{\fuu}\uu^0_{s}\|^2_{L^2(D)},\\
&\leq A(C^*+2\sqrt{2})\|\exp Z_s\|_{L^2(D)}\|\partial^{\fuu}\uu^0_{s}\|_{L^2(D)}|\fuu|!2^{|\fuu|}S^J_{\fuu}+\|\partial^{\fuu}\uu^0_{s}\|^2_{L^2(D)}\\
& \leq \tfrac{1}{2}\|\partial^{\fuu}\uu^0_{s}\|_{L^2(D)}|\fuu|!2^{|\fuu|}S^J_{\fuu}+\|\partial^{\fuu}\uu^0_{s}\|^2_{L^2(D)},
\end{align*}
using the inequality $a!b!\leq (a+b-1)!$ for any $a,b\in\mathbb{N}$.
Solving the quadratic inequality for $S^J_{\fuu}$ yields the claim \eqref{desire_mid}:
\[S^J_{\fuu}\leq \tfrac{1}{2}\|\partial^{\fuu}\uu^0_{s}\|_{L^2(D)}|\fuu|!2^{|\fuu|}+\sqrt{2}\|\partial^{\fuu}\uu^0_{s}\|_{L^2(D)}\leq |\fuu|!2^{|\fuu|}\|\partial^{\fuu}\uu^0_{s}\|_{L^2(D)}.\]
Last, substituting \eqref{desire_mid} into \eqref{mid_mid} and repeating the above argument yield
\begin{equation*}
\|\partial^{\fuu}\uu^J_{s,h}\|^2_{L^2(D)}\leq\tfrac{1}{2}\|\partial^{\fuu}\uu^0_{s}\|^2_{L^2(D)}\left(|\fuu|!\right)^2 2^{2|\fuu|}+\|\partial^{\fuu}\uu^0_{s}\|^2_{L^2(D)}.
\end{equation*}
This, together with \eqref{mid_init_est}, implies
\begin{equation*}
\|\partial^{\fuu}\uu^J_{s,h}\|_{L^2(D)}\leq 2|\fuu|! 2^{|\fuu|} \|\partial^{\fuu}\uu^0_{s}\|_{L^2(D)}\leq (2\sqrt{2}+2C^*)|\fuu|! 2^{|\fuu|}\|\exp Z_s\|_{H^1(D)}.
\end{equation*}
This completes the proof of the theorem.
\end{proof}

\subsection{A function space setting in $\R^s$}

{{Most QMC methods in the literature are applied and analyzed with the class of functions defined over the unit hypercube. In this work, the suitable function spaces in the QMC analysis are weighted Sobolev spaces on $(0,1)^s$, which consist of functions with square-integrable mixed first derivatives. To be more precise, it is known from the classical theory that if the integrand lies in a suitable weighted Sobolev space, randomly-shifted lattice rules (RSLRs) can be constructed so as to (nearly)  achieve the optimal convergence rate $\mathcal{O}(n^{-1})$; see \cite{qmc_ref_1, qmc_ref_2, qmc_ref_3, qmc_ref_4, qmc_ref_5} and recent surveys \cite{qmc_ref_6, qmc_ref_7}.}} Since the integral in \eqref{QMC_term} is defined over $\R^s$, we have first transformed the domain to $(0,1)^s$ and obtain $F^J_h(\boldsymbol{\Phi}_s^{-1}(\cdot))$. However, the resulting integrand may be unbounded near the boundary $\partial(0,1)^s$,
and thus the standard QMC theory is not directly applicable. The function space setting for the integral of the type \eqref{QMC_term} has been studied in \cite{qmc_ref_8, qmc_ref_9, qmc_ref_10, qmc_ref_11, qmc_ref_14, qmc_ref_13}, and the optimal convergence rate has been obtained using RSLRs. The corresponding weighted Sobolev norm is defined by \cite{qmc_ref_14,qmc_ref_13}
\begin{equation}\label{wss}
\|F\|^2_{\mathcal{W}_s}:=\sum_{\fuu\subset[s]}\frac{1}{\gamma_{\fuu}}\int_{\R^{|\fuu|}}\!\!\!\!\Big(\int_{\R^{s-|\fuu|}}\!\!\!\!\!\!\!\!\partial^{\fuu}
F(\yy_{\fuu};\yy_{[s]\setminus\fuu})\prod_{j\in [s]\setminus\fuu}\!\!\!\!\phi(y_j)\dy_{[s]\setminus\fuu}\Big)^2\prod_{j\in\fuu}\psi^2_j(y_j)\dy_{\fuu},
\end{equation}
where $[s]=\{1,2,\cdots,s\}$, $\partial^{\fuu}F$ denotes the mixed first derivative with respect to each ``active'' variables $y_j$ for $j\in\fuu$, and $\yy_{[s]\setminus\fuu}$ is the ``inactive'' variable $y_j$, $j\notin\fuu$. The norm \eqref{wss} is called ``unanchored'' since the inactive variables are integrated out as opposed to being ``anchored'' at a certain fixed value, e.g., $0$.

For each $j\geq1$, the continuous weight function $\psi_j:\R\rightarrow\R_+$ in \eqref{wss} should be
properly chosen in order to handle the singularities for the active variables. The analysis in \cite{qmc_ref_13} indicates
that $\psi^2_j(y)$ should decay slower than the standard Gaussian density in \eqref{QMC_term} as $|y|\rightarrow\infty$: for some $a_j>0$, such that
\begin{equation}\label{active_weight}
\psi^2_j(y)=\exp(-2a_j|y|).
\end{equation}
Throughout, we assume that for some constants $0<a_{\min}< a_{\max}<\infty$,
\begin{equation}\label{a_cond}
a_{\min}<a_j\leq a_{\max},\quad j\in\mathbb{N}.
\end{equation}

For each index $\fuu\subset\mathbb{N}$ with finite cardinality $|\fuu|<\infty$, we associate
a weight parameter vector $\gamma_{\fuu}>0$ to indicate the relative importance
of the variables. We write $\boldsymbol{\gamma}=(\gamma_{\fuu})_{\fuu\subset\mathbb{N}}$ and
let $\gamma_{\emptyset}:=1$. In \cite{qmc_ref_14}, only ``products weights'' were considered,
i.e., there exists a sequence $\gamma_1\geq\gamma_2\geq\cdots>0$, where each $\gamma_j$ is
associated with an integral variable $y_j$ and let $\gamma_{\fuu}:=\prod_{j\in\fuu}\gamma_j$.
See also \cite{qmc_ref_13} for further generalizations.
%The results of \cite{qmc_ref_14} were further generalized in \cite{qmc_ref_13}, where the weight parameters depend on the parametric dimension $s\in\mathbb{N}$.
The choice of the weight parameters $\gamma_{\fuu}$ is important to guarantee that the constant
in the QMC error bound does not grow exponentially as $s\rightarrow\infty$. In this work,
we employ the so-called ``product and order dependent weights'' (``POD'' weights) introduced
in \cite{KSS2012}. Specifically, we consider two different sequences $\Gamma_0=\Gamma_1=1$,
$\Gamma_2,\cdots, \Gamma_s$ and $\gamma_1\geq\gamma_2\geq\cdots\geq\gamma_s>0$ such that
$\gamma_{\fuu}:=\Gamma_{|\fuu|}\prod_{j\in\fuu}\gamma_j$.

\subsection{Error bound for randomly-shifted lattice rules}
To bound the error of the QMC integration, we define the worst-case error $e^{\rm{wor}}_{s,N}(\boldsymbol{z},\boldsymbol{\Delta})$ of
the shifted lattice rule \eqref{QMC_main}: for a generating vector $\boldsymbol{z}$ and a random shift $\boldsymbol{\Delta}$,
$e^{\rm{wor}}_{s,N}(\boldsymbol{z},\boldsymbol{\Delta}):=\sup_{\|F\|_{\mathcal{W}_s}\leq1}|I_s(F)-\mathcal{Q}_{s,N}(F;\boldsymbol{\Delta})|$,
where $I_s$ is the integral defined in \eqref{QMC_term}.
By the linearity of integration, we have
\begin{equation*}%\label{error_bound_1}
|I_s(F^J_{s,h})-\mathcal{Q}_{s,N}(F^J_{s,h};\boldsymbol{\Delta})|\leq e^{\rm{wor}}_{s,N}(\boldsymbol{z},\boldsymbol{\Delta})\|F^J_{s,h}\|_{\mathcal{W}_s}.
\end{equation*}
In this work, we consider the root-mean-square error (RMSE) for QMC:
\begin{equation}\label{error_bound_2}
%\sqrt{\mathbb{E}^{\boldsymbol{\Delta}}|I_s(F^J_{s,h})-\mathcal{Q}_{s,N}(F^J_{s,h};\cdot)|^2}
{\rm RMSE}_{\rm qmc}(N)\leq e^{\rm{sh}}_{s,N}(\boldsymbol{z})\|F^J_{s,h}\|_{\mathcal{W}_s},
\end{equation}
with $e^{\rm{sh}}_{s,N}(\boldsymbol{z}):=(\int_{[0,1]^s}(e^{\rm{wor}}_{s,N}(\boldsymbol{z}, \boldsymbol{\Delta}))^2\,{\rm{d}}\boldsymbol{\Delta})^{1/2}$ (often known as shift-averaged worst case error). By \eqref{error_bound_2}, we can decouple the dependence on $\boldsymbol{z}$ from that on the integrand $F^J_{s,h}$.

A generating vector $\boldsymbol{z}=(z_1,z_2,z_3,\cdots)$ can be constructed by a component-by-component algorithm which determines $z_1,z_2,z_3,\cdots$ sequentially, using $e^{\rm{sh}}_{s,N}(\boldsymbol{z})$ as the search criterion: given that $z_1,\cdots,z_i$ are already determined, $z_{i+1}$ is chosen from the set $\{1\leq z\leq N-1:{\rm{gcd}}(z,N)=1\}$ to minimize $ e^{\rm{sh}}_{i+1,N}(z_1,\cdots,z_{i+1})$. See \cite{qmc_ref_13} for the precise formula for $e^{\rm{sh}}_{s,N}(\boldsymbol{z})$ for general weight functions $\psi_j$ and weight parameters $\gamma_{\fuu}$:
\[
\left(e^{\rm{sh}}_{s,N}(\boldsymbol{z})\right)^2=\sum_{\emptyset\neq\fuu\subset[s]}\frac{\gamma_{\fuu}}{N}\sum^N_{i=1}\prod_{k\in\fuu}\theta_k\left(\left\{\frac{iz_k}{N}\right\}\right),
\]
with
\[
\theta_k(f)=\int^{\infty}_{\Phi^{-1}(f)}\frac{\Phi(t)-f}{\psi^2_k(t)}\dt + \int^{\infty}_{\Phi^{-1}(1-f)}\frac{\Phi(t)-1+f}{\psi^2_k(t)}\dt-\int^{\infty}_{-\infty}\frac{\Phi^2(t)}{\psi^2_k(t)}\dt.
\]
See also \cite{qmc_ref_13} for the following choices of $\phi$ and $\psi_k$ with
a nearly $\mathcal{O}(N^{-1})$ convergence rate: the cumulative normal distribution function and weight functions
in \eqref{active_weight}. Below we use a result from \cite{main_ref}. Note here that $\varphi(p)=p-1$ for $p$ prime, and it is known that $\frac{1}{\varphi(N)}<\frac{9}{N}$ for any $N\leq 10^{30}$. Hence, in practice, we can replace $\varphi(N)$ by $\frac{C}{N}$ for some $C>0$.

%Henceforth, we shall continue to assume the smallness of initial data \eqref{eqn:smallness}.
\begin{theorem}\label{mse_1}
For given $h>0$, $s,J,N\in\mathbb{N}$, weight parameters $\boldsymbol{\gamma}=(\gamma_{\fuu})_{\fuu\subset\mathbb{N}}$, Gaussian density function $\phi$ and weight function $\psi_j$ defined in \eqref{active_weight}, an RSLR with $N$ points can be constructed by a component-by-component algorithm satisfying for any $\lambda\in(1/2,1]$,
\begin{equation*}%\label{thm_est_1}
%\sqrt{\mathbb{E}^{\boldsymbol{\Delta}}|I_s(F^J_{s,h})-\mathcal{Q}_{s,N}(F^J_{s,h};\boldsymbol{\Delta})|^2}
{\rm RMSE}_{\rm qmc}(N)\leq
\Big(\sum_{\emptyset\neq\fuu\subset\{1:s\}}\gamma^{\lambda}_{\fuu}\prod_{i\in\fuu}\varrho_i(\lambda)\Big)^{\frac{1}{2\lambda}}\left[\varphi(N)\right]^{-\frac{1}{2\lambda}}\|F^J_{s,h}\|_{\mathcal{W}_s}
\end{equation*}
with
\begin{equation}\label{thm_est_2}
\varrho_i(\lambda):=2\left(\frac{\sqrt{2\pi}\exp(a^2_i/\eta)}{\pi^{2-2\eta}(1-\eta)\eta}\right)^{\lambda}\zeta\left(\lambda+\frac{1}{2}\right)
\quad{\rm{and}}\quad\eta:=\frac{2\lambda-1}{4\lambda},
\end{equation}
where $\varphi(n):=|\{1\leq z\leq n-1:{\rm{gcd}}(z,n)=1\}|$ is the Euler totient function and $\zeta(x):=\sum^{\infty}_{k=1}k^{-x}$ denotes the Riemann zeta function.
\end{theorem}

\subsection{Estimate for the weighted Sobolev norm}
The next result shows that for each $s\in\mathbb{N}$ and choice of $\gamma_{\fuu}$, we have $\|F^J_{s,h}\|_{\mathcal{W}_s}<\infty$ for all $J\in\mathbb{N}$ and $h>0$. Thus, together with Theorem \ref{mse_1}, we obtain an estimate for the $\rm RMSE_{qmc}$ with a nearly $\mathcal{O}(N^{-1})$ convergence rate. However, the bound may depend on the parametric dimension $s\in\mathbb{N}$. In Section \ref{ssec:weight}, we will prove that a careful choice of $\gamma_{\fuu}$ can remove the dependency on $s\in\mathbb{N}$.

\begin{theorem}\label{mse_2}
Let  the weight functions $\psi_j$ be defined in \eqref{active_weight} and $F^J_{s,h}$ the
integrand in \eqref{QMC_term} for any $J=1,\cdots,\ell$ and $h>0$. Then $F^J_{s,h}\in \mathcal{W}_s$,
and with constants $a_i$ and $b_i$ defined in \eqref{active_weight} and \eqref{eqn:b}, its $\mathcal{W}_s$ norm is bounded by
\begin{equation*}%\label{thm_est_3}
\|F^J_{s,h}\|^2_{\mathcal{W}_s}\leq(2\sqrt{2}+2C^*)^2
\|\mathcal{G}\|_{(L^2(D))'}^2\sum_{\fuu\subset[s]}\frac{(|\fuu|!)^22^{2|\fuu|}}
{\gamma_{\fuu}}\prod_{i\in\fuu}\frac{b_i^2}{a_i}.
\end{equation*}
\end{theorem}
\begin{proof}
By Theorem \ref{reg_estimate} and the linearity of $\mathcal{G}$, we obtain for each $\yy\in\R^s$,
\[
|\partial^{\fuu}F^J_{s,h}(\yy)|\leq\|\mathcal{G}\|_{(L^2(D))'}\|\partial^{\fuu}\uu^J_{s,h}(\yy)\|_2\leq
%(2\sqrt{2}+2C^*)\|\mathcal{G}\|_{(L^2(D))'}
\widetilde{C}|\fuu|!2^{|\fuu|}\bigg(\prod_{j\in\fuu}b_j\bigg)\|\exp Z_s(\yy)\|_{H^1(D)},
\]
with $\widetilde{C}:=(2\sqrt{2}+2C^*)\|\mathcal{G}\|_{(L^2(D))'}$. Further, the smallness of the truncated log-normal random field \eqref{trunc_small} implies that $\|\exp Z_s(\yy)\|_{H^1(D)}\leq1$ for all $s\in\mathbb{N}$ and $\yy\in\R^s$. Then the definition of the weighted Sobolev norm \eqref{wss} implies
\begin{align*}
\|F^J_{s,h}\|_{\mathcal{W}_s}^2
\leq& \widetilde{C}^2\sum_{\fuu\subset[s]}\frac{(|\fuu|!)^22^{2|\fuu|}}{\gamma_{\fuu}}\!\!\Big(\prod_{j\in\fuu}b_j\Big)^2 \int_{\R^{|\fuu|}}\!\!\Big(\int_{\R^{s-|\fuu|}}\!\!\!\!\prod_{j\in[s]\setminus\fuu}\!\!\phi(y_j)\dy_{[s]\setminus\fuu}\Big)^2\prod_{j\in\fuu}\psi^2_j(y_j)\dy_{\fuu}\\
\leq& \widetilde{C}^2\sum_{\fuu\subset[s]}\frac{(|\fuu|!)^22^{2|\fuu|}}
{\gamma_{\fuu}}\prod_{j\in\fuu}\frac{b_j^2}{a_j},
\end{align*}
since $\int_{\R}\phi(y)\,{\rm{d}}y=1$ and $\int_{\R}\psi^2_j(y)\,{\rm{d}}y=\frac{1}{a_j}$ for all $j\in\mathbb{N}$.
\end{proof}

Theorems \ref{mse_1} and \ref{mse_2} together give the following RMSE estimate.
\begin{theorem}\label{mse_3}
Let $F^J_{s,h}$ be the integrand defined in \eqref{QMC_term} and let $\psi_j$ be a weight function defined in \eqref{active_weight}. For given $s,J,N\in\mathbb{N}$ with $N\leq 10^{30}$, $h>0$, weights $\boldsymbol{\gamma}=(\gamma_{\fuu})_{\fuu\subset\mathbb{N}}$ and standard Gaussian density function $\phi$, we can construct an RSLR with $N$ points in $s$ dimensions by a component-by-component algorithm such that for any $\lambda\in(1/2,1]$,
\begin{equation}\label{thm_est_4}
%\sqrt{\mathbb{E}^{\boldsymbol{\Delta}}|I_s(F^J_{s,h})-\mathcal{Q}_{s,N}(F^J_{s,h};\boldsymbol{\Delta})|^2}
{\rm RMSE}_{\rm qmc}(N)\leq 9(2\sqrt{2}+2C^*)\|\mathcal{G}\|_{(L^2(D))'}K_{\boldsymbol{\gamma},s}(\lambda)N^{-\frac{1}{2\lambda}},
\end{equation}
with $\varrho_i(\lambda)$ defined in \eqref{thm_est_2},
\begin{equation}\label{thm_est_5}
K_{\boldsymbol{\gamma},s}(\lambda):=\Big(\sum_{\emptyset\neq\fuu\subset[s]}\gamma^{\lambda}_{\fuu}\prod_{i\in\fuu}\varrho_i(\lambda) \Big)^{\frac{1}{2\lambda}} \Big(\sum_{\fuu\subset[s]}\frac{(|\fuu|!)^22^{2|\fuu|}}
{\gamma_{\fuu}}\prod_{i\in\fuu}\frac{b_i^2}{a_i}\Big)^{1/2}.
\end{equation}
\end{theorem}

In general, $K_{\boldsymbol{\gamma},s}(\lambda)$ may grow with the parametric dimension $s$. In order to bound $K_{\boldsymbol{\gamma},s}(\lambda)$ uniformly with respect to $s\in\mathbb{N}$, we need to choose $\gamma_{\fuu}$ carefully so that
\begin{equation}\label{thm_est_6}
K_{\boldsymbol{\gamma}}(\lambda):=\Big(\sum_{|\fuu|<\infty}\gamma^{\lambda}_{\fuu}\prod_{i\in\fuu}\varrho_i(\lambda)\Big)^{\frac{1}{2\lambda}} \Big(\sum_{|\fuu|<\infty}\frac{(|\fuu|!)^22^{2|\fuu|}}{\gamma_{\fuu}}\prod_{i\in\fuu}\frac{b_i^2}{a_i}\Big)^{1/2}<\infty.
\end{equation}
If \eqref{thm_est_6} holds, then the estimate $K_{\boldsymbol{\gamma},s}(\lambda)\leq K_{\boldsymbol{\gamma}}(\lambda)<\infty$ holds for any $s\in\mathbb{N}$, and accordingly the error bound \eqref{thm_est_4} is independent of the dimension $s$.

\subsection{Choice of weight parameters $\gamma_{\fuu}$}\label{ssec:weight}
For any $\lambda\in(1/2,1]$, we follow the strategy in \cite{KSS2012,main_ref} to choose weight parameters $\gamma_{\fuu}$ that minimize the constant $K_{\boldsymbol{\gamma}}(\lambda)$, cf. \eqref{thm_est_6}, and that $K_{\boldsymbol{\gamma}}(\lambda)$ is finite.
Note that the constant $K_{\boldsymbol{\gamma},s}(\lambda)$ in \eqref{thm_est_5} and the uniform bound $K_{\boldsymbol{\gamma}}(\lambda)$ in \eqref{thm_est_6} have the same form as the function appearing in Lemma \ref{aux_conv_lemma_1}. Therefore, we can infer the proper form of the weight parameters $\gamma_{\fuu}$. Below we specify the parameter $\lambda>0$ so that the constant $K_{\boldsymbol{\gamma}}(\lambda)$ is finite in our setting and to obtain a good convergence rate.

\begin{theorem}\label{QMC_final}
Let $\psi_j$ be the weight functions defined in \eqref{active_weight} with $a_j$ satisfying \eqref{a_cond}, and Assumption \ref{ass_indiv} hold for some $p\leq1$. When $p=1$, additionally,
\begin{equation}\label{thm_further_ass}
\sum_{j\geq1}b_j\leq \frac{1}{2}\sqrt{\frac{a_{\min}}{\varrho_{\max}(1)}},
\end{equation}
with $\varrho_{\max}(\lambda)$ defined by replacing $a_i$ in \eqref{thm_est_2} by $a_{\max}$ in \eqref{a_cond}. Then for each fixed $\lambda\in(1/2,1]$, the weight
\begin{equation}\label{final_weight}
\gamma_{\fuu}=\gamma^*_{\fuu}(\lambda):=\Big((|\fuu|!)^22^{2|\fuu|}\prod_{j\in\fuu}\frac{b_j^2}{a_j\varrho_j(\lambda)}\Big)^{1/(1+\lambda)}
\end{equation}
is the minimizer of $K_{\boldsymbol{\gamma}}(\lambda)$ if the minimum is finite. Additionally, if we choose
\begin{equation}\label{final_lam}
\lambda=\lambda_*:=
\begin{cases}
\frac{1}{2-2\delta}, &{\rm{if}}\,\, p \in (0, 2/3],\\
\frac{p}{2-p}, &{\rm{if}}\,\, p \in (2/3, 1),\\
1, &{\rm{if}}\,\, p =1,
\end{cases}
\end{equation}
for arbitrary $\delta\in(0,1/2]$, and set $\gamma_{\fuu}=\gamma_{\fuu}^*(\lambda_*)$, then $K_{\boldsymbol{\gamma}}(\lambda)<\infty$. Moreover, an RSLR can be constructed by a component-by-component algorithm such that
\[
%\sqrt{\mathbb{E}^{\boldsymbol{\Delta}}|I_s(F^J_{s,h})-\mathcal{Q}_{s,N}(F^J_{s,h};\boldsymbol{\Delta})|^2}
\mathrm{RMSE}_{\rm qmc}\lesssim N^{-\chi},
\quad \mbox{with }
\chi = (2\lambda_*)^{-1},
%\begin{cases}
%{1-\delta}\quad&{\rm{if}}\,\, p \in (0, 2/3],\\
%{1/p-1/2}\quad &{\rm{if}}\,\, p \in (2/3, 1),\\
%{1}/{2}\quad &{\rm{if}}\,\, p =1,
%\end{cases}
%\begin{cases}
%N^{-(1-\delta)}\quad&{\rm{if}}\,\, p \in (0, 2/3],\\
%N^{-(1/p-1/2)}\quad &{\rm{if}}\,\, p \in (2/3, 1),\\
%N^{-\frac{1}{2}}\quad &{\rm{if}}\,\, p =1,
%\end{cases}
\]
in other words,
\[
\mathrm{RMSE}_{\rm qmc}\lesssim
\begin{cases}
N^{-(1-\delta)}\quad&{\rm{if}}\,\, p \in (0, 2/3],\\
N^{-(1/p-1/2)}\quad &{\rm{if}}\,\, p \in (2/3, 1),\\
N^{-\frac{1}{2}}\quad &{\rm{if}}\,\, p =1,
\end{cases}
\]
where the constants are independent of the truncation dimension $s\in\mathbb{N}$, but may depend on $p\in(0,1]$ and $\delta\in(0,1/2]$ if relevant.
\end{theorem}

\begin{proof}
First, note that the finite subsets of $\mathbb{N}$ in \eqref{thm_est_6} can be ordered and the specific choice of ordering is not important, since the convergence is unconditional. Therefore, by Lemma \ref{aux_conv_lemma_1}, the choice of weights \eqref{final_weight} minimizes $K_{\boldsymbol{\gamma}}(\lambda)$, cf \cite{KSS2012, main_ref}. Next, we prove that $K_{\boldsymbol{\gamma}}(\lambda)$ is finite when the weight and the parameter $\lambda$ are given by \eqref{final_weight} and \eqref{final_lam}, respectively. Indeed, define the following quantity
\begin{equation}\label{S_def}
S_{\lambda}:=\sum_{|\fuu|<\infty}(\gamma_{\fuu}^*)^{\lambda}\prod_{j\in\fuu}\varrho_j(\lambda)=\sum_{|\fuu|<\infty} \Big((|\fuu|!)^22^{2|\fuu|}\prod_{j\in\fuu}\frac{[\varrho_j(\lambda)]^{1/\lambda}b_j^2}{a_j}\Big)^{\frac{\lambda}{1+\lambda}}.
\end{equation}
Then $K_{\boldsymbol{\gamma}}(\lambda)=S_{\lambda}^{1/(2\lambda)+1/2}$, and hence it suffices to show that $S_{\lambda}$ is finite.
From \eqref{thm_est_2}, for each $\lambda$, $\varrho_j(\lambda)$ monotonically increases with respect to $a_j$, and thus we obtain $\varrho_j(\lambda)\leq\varrho_{\max}(\lambda)$ for any $j\geq1$. Consequently,
\begin{equation}\label{quant_mid_est_1}
S_{\lambda}\leq\sum_{|\fuu|<\infty}(|\fuu|!)^{2\lambda/(1+\lambda)}\prod_{j\in\fuu}\left(\frac{4[\varrho_{\max}(\lambda)]^{1/\lambda}}{a_{\min}}b_j^2\right)^{\lambda/(1+\lambda)}.
\end{equation}
Now we consider the cases $\lambda\in(1/2,1)$ and $\lambda=1$ separately. For $\lambda\in(1/2,1)$, we have
$2\lambda/(1+\lambda)<1$. Next, we multiply and divide the right-hand side of \eqref{quant_mid_est_1} by $\prod_{j\in\fuu}A_j^{2\lambda/(1+\lambda)}$, with $A_j>0$ to be specified. By H\"older's inequality with H\"older conjugate exponents $(1+\lambda)/(2\lambda)$ and $(1+\lambda)/(1-\lambda)$, we obtain
\begin{align*}
S_{\lambda}
&\leq\sum_{|\fuu|<\infty}(|\fuu|!)^{2\lambda/(1+\lambda)}\prod_{j\in\fuu}A_j^{2\lambda/(1+\lambda)}\prod_{j\in\fuu}\left(\frac{4[\varrho_{\max}(\lambda)]^{1/\lambda}}{a_{\min}}\frac{b_j^2}{A_j^2}\right)^{\lambda/(1+\lambda)}\\
&\leq\left(\sum_{|\fuu|<\infty}|\fuu|!\prod_{j\in\fuu}A_j\right)^{2\lambda/(1+\lambda)}\left(\sum_{|\fuu|<\infty}\prod_{j\in\fuu}\left(\frac{4[\varrho_{\max}(\lambda)]^{1/\lambda}}{a_{\min}}\frac{b_j^2}{A_j^2}\right)^{\lambda/(1-\lambda)}\right)^{(1-\lambda)/(1+\lambda)}\\
&\leq \left(\frac{1}{1-\sum_{j\geq1}A_j}\right)^{\frac{2\lambda}{1+\lambda}}\exp \left(\frac{1-\lambda}{1+\lambda}\left(\frac{4[\varrho_{\max}(\lambda)]^{1/\lambda}}{a_{\min}}\right)^{\frac{\lambda}{1-\lambda}}\sum_{j\geq1} \left(\frac{b_j}{A_j}\right)^{\frac{2\lambda}{1-\lambda}}\right).
\end{align*}
where the last step follows from Lemma \ref{aux_conv_lemma_2}, which holds if
\begin{equation}\label{new_cond_1}
\sum_{j\geq1}A_j<1\quad{\rm{and}}\quad\sum_{j\geq1}(b_j/A_j)^{2\lambda/(1-\lambda)}<\infty.
\end{equation}
Upon choosing
$A_j:=b_j^p/\alpha$ for some $\alpha>\sum_{j\geq1}b_j^p$, Assumption \ref{ass_indiv} implies $\sum_{j\geq1}A_j<1$. Further, Assumption \ref{ass_indiv} also implies $\sum_{j\geq1}b_j^q$ for any $q\geq p$. Hence the second sum in \eqref{new_cond_1} converges provided that
$\frac{2\lambda}{1-\lambda}(1-p)\geq p$ if and only if $\lambda\geq\frac{p}{2-p}$.
Since $\lambda\in(1/2,1)$, if $p\in(0,2/3)$ we have $\frac{1}{2}\geq\frac{p}{2-p}$ and hence we can choose $\lambda=1/(2-2\delta)$ for some $\delta\in(0,1/2)$, so that $\frac{p}{2-p}\leq\frac{1}{2}<\lambda<1$. If $p\in(2/3,1)$, we have $\frac{1}{2}<\frac{p}{2-p}<1$, and we may choose $\lambda=p/(2-p)$. When $p=1$, we choose $\lambda=1$. Then by Lemma \ref{aux_conv_lemma_2}, we have from \eqref{quant_mid_est_1} that
\[
S_1\leq\sum_{|\fuu|<\infty}|\fuu|!\prod_{j\in\fuu}\left(\frac{4\varrho_{\max}(1)}{a_{\min}}b_j^2\right)^{\frac{1}{2}} \leq\Big(1-\sum_{j\geq1}2b_j\Big(\frac{\varrho_{\max}(1)}{a_{\min}}\Big)^{\frac1{2}}\Big)^{-1},
\]
which is finite by assumption \eqref{thm_further_ass}. This completes the proof of the theorem.
\end{proof}

Last, we give a proper choice of $a_j$ which will be used in numerical experiments.
\begin{corollary}\label{aj_choice}
Let $\lambda=\lambda_*$ and $\gamma_{\fuu}=\gamma_{\fuu}^*(\lambda_*)$ be defined in \eqref{final_lam} and \eqref{final_weight}, respectively. Then the constant $K_{\boldsymbol{\gamma}}(\lambda)$ in \eqref{thm_est_6} is minimized when
$a_j=\sqrt{\frac{2\lambda_*-1}{8\lambda_*}}$ for all $j\geq1$.
\end{corollary}
\begin{proof}
In the proof of Theorem \ref{QMC_final}, we have $K_{\boldsymbol{\gamma}}(\lambda)=S_{\lambda}^{1/(2\lambda)+1/2}$, with $S_{\lambda}$ defined in \eqref{S_def}. Since all terms in \eqref{S_def} are positive, it is sufficient to minimize each $[\varrho_j(\lambda)]^{1/\lambda}/a_j$ with respect to $a_j$, in order to minimize $K_{\boldsymbol{\gamma}}(\lambda)$ with respect to $a_j$. By definition, we have $[\varrho_j(\lambda)]^{1/\lambda}=c\exp(a_j^2/\eta_*)$ for some $c>0$ independent of $a_j$ and $\eta_*=\frac{1}{2}-\frac{1}{4\lambda}$. Then direct computation shows that the choice of $a_j$ minimizes $S_{\lambda}$, and hence $K_{\boldsymbol{\gamma}}(\lambda)$.
\end{proof}

Last, we give the total error of the approximation $\mathcal{Q}_{s,N}(F^J_{s,h};\boldsymbol{\Delta})$.
\begin{theorem}\label{final_main_thm}
Let the assumptions of Theorems \ref{fem_main_thm}, \ref{truncation_thm} and \ref{QMC_final} hold. Then the RMSE of $\mathcal{Q}_{s,N}(F^J_{s,h};\boldsymbol{\Delta})$ with respect to the random shift $\boldsymbol{\Delta}\in[0,1]^s$ is bounded by
\begin{equation}\label{final_est_com}
(\mathbb{E}^{\boldsymbol{\Delta}}[(\mathbb{E}[\mathcal{G}(\uu(t_J))]-\mathcal{Q}_{s,N}(F^J_{s,h};\boldsymbol{\Delta}))^2])^{1/2}\leq C(t_J^{-\frac{1}{2}}h^2+\tau + s^{-\frac{k}{2}+\frac{3}{2}}+N^{-\chi}),
\end{equation}
%with arbitrary $\delta\in(0,1/2]$ and
%\[
%\chi =
%\begin{cases}
%{-(1-\delta)}\quad&{\rm{if}}\,\, p \in (0, 2/3],\\
%{-(1/p-1/2)}\quad &{\rm{if}}\,\, p \in (2/3, 1),\\
%{-\frac{1}{2}}\quad &{\rm{if}}\,\, p =1\,\,{\rm{and}}\,\,\eqref{thm_further_ass}\,\,{\rm{holds}}.
%\end{cases}
%\]
where the constant $C>0$ is independent of positive parameters $h$, $\tau$, $s$ and $N$.
\end{theorem}
\begin{proof}
The  error $\mathbb{E}\left[\mathcal{G}(\uu(t_J))\right]-\mathcal{Q}_{s,N}(F^J_{s,h};\boldsymbol{\Delta})$ can be decomposed into
$$\mathbb{E}[\mathcal{G}(\uu(t_J))-\mathcal{G}(\uu^J_h)]
+ \mathbb{E}[\mathcal{G}(\uu^J_h)-\mathcal{G}(\uu^J_{s,h})]
+(\mathbb{E}[\mathcal{G}(\uu^J_{s,h})]-\mathcal{Q}_{s,N}(F^J_{s,h};\boldsymbol{\Delta})),$$
where the expectation $\mathbb{E}$ is with respect to $\yy\in U_{\bb}$.
Then the desired bound on $(\mathbb{E}^{\boldsymbol{\Delta}}[(\mathbb{E}[\mathcal{G}(\uu(t_J))]-\mathcal{Q}_{s,N}(F^J_{s,h};\boldsymbol{\Delta}))^2])^{1/2}$ follows from
Theorems \ref{fem_main_thm}, \ref{truncation_thm} and \ref{QMC_final} directly.
\end{proof}

\section{Numerical experiments}
\label{sec_exp}
Since Theorem \ref{QMC_final} is our main novel theoretical finding, we only conduct several numerical tests to complement the analysis in Theorem \ref{QMC_final}, we present some numerical results for the system \eqref{main_eq1}-\eqref{main_ic} in the domain $[0,T]\times D\times\Omega$, with $T=1$ and $D=(0,1)^2$. All computations were performed using MATLAB on HPC system at The University of Hong Kong, using 64 cores, each with 3GB memory.

In the computation, we take a time step size $\tau=0.1$. We divide the square domain $D$ into $1/h^2$ congruent small squares (with $h=1/16$), and then further divide each small square into two right triangles, obtaining a shape-regular triangulation $\mathcal{T}_h$. To solve problem \eqref{eqn:WF}, we employ the Taylor-Hood element on the mesh $\mathcal{T}_h$, i.e., a conforming piecewise quadratic element for the velocity $\uu$ and a conforming piecewise linear element for the pressure $p$. The resulting FEM spaces are given by
\begin{align*}
H_h
&:=\{\mathbf{W}\in C(\overline{D})^2:\mathbf{W}|_K\in {P}_2(K)^2\,\,\forall K\in\mathcal{T}_h\,\,{\rm{and}}\,\,\mathbf{W}|_{\partial D}=\boldsymbol{0}\},\\
Q_h
&:=\{\Pi\in C(\overline{D}):\Pi|_K\in {P}_1(K)\,\,\forall K\in\mathcal{T}_h  \text{ and } \int_D \Pi \dx=0 \}.
\end{align*}
This choice satisfies the discrete inf-sup condition \eqref{eqn:inf-sup}.
To handle the divergence-free subspace, we employ a continuous bilinear form $c(\cdot,\cdot)$ on $H_h\times Q_h$ by
$c(\vv,q)=-\int_D q \,(\nabla\cdot \vv) \dx$. The trilinear term $B[\uu,\vv,\ww]$ \eqref{SIBE}, arising from the convective term $(\uu\cdot \nabla)\uu$, is nonlinear. We employ Picard's method to linearize the problem, which constructs a sequence of approximations $(\uu^{j+1,k}_{s,h},p^{j+1,k}_{s,h})$ by solving
\begin{equation*}
\left\{\begin{split}
\tau^{-1}(\uu^{j+1,k}_{s,h}&-\uu^{j}_{s,h},\vv_h)+a(\uu^{j+1,k}_{s,h},\vv_h)\\
     &+B[\uu^{j+1,k-1}_{s,h},\uu^{j+1,k}_{s,h},\vv_h]+c(\vv_h,p^{j+1,k}_{s,h})=0, \quad \forall \vv_h\in H_h,\\
c(\uu^{j+1,k}_{s,h},q_h)&=0,\quad\forall q_h\in Q_h.
%\uu_{s,h}^{j+1,k}&=\boldsymbol{0},\quad{\rm{on}}\,\,\partial D.
\end{split}\right.
\end{equation*}
This formulation is equivalent to the pressure-free formulation for the FEM approximation in Section \ref{sec_fem}, since we use an inf-sup stable velocity-pressure pair for the discrete solutions. We take the initialization
$\uu^{j+1,0}_{s,h}=\uu^{j}_{s,h}$. The iteration process is terminated when the relative $L^2(D)$-norm between the solutions of the current iteration and the previous one falls below a given tolerance $\eta=10^{-7}$. Once the iteration converges after $K$ steps, we assign $\uu_{s,h}^{j+1}=\uu_{s,h}^{j+1,K}$ and $p_{s,h}^{j+1}=p_{s,h}^{j+1,K}$.

%To solve the system, we apply a direct solver (LU) since the number of degrees of freedom in the spatial domain is not large.

%The uncertainty arises in the initial conditions. The uncertainty of
%input data was given in $\uu_{s,h}^0$ which is the $L^2$-projection of $\uu^0_{s}$ into $V_h$.

The random initial data $\uu_{s,h}^0$ is the $L^2(D)$-projection  into $V_h$ of $\uu^0_{s}$, the truncated KL expansions with $s=400$ of the initial condition $\uu^{0}$ in \eqref{lognormal_initial}, which arises from the random field $Z$ with a zero mean and  Mat\'ern covariance function $\rho_\nu(|\bold{x}-\bold{x}'|)$, with $r=|\bold{x}-\bold{x}'|$ for all $\bold{x},\bold{x}'\in D$, defined by
\begin{equation*}
    \rho_\nu(r)=\sigma^2\frac{2^{1-\nu}}{\Gamma(\nu)}\left(2\sqrt{\nu}\frac{r}{\lambda_{\mathrm{C}}}\right)^\nu K_\nu\left(2\sqrt{\nu}\frac{r}{\lambda_{\mathrm{C}}}\right).
\end{equation*}
Here, $\Gamma$ is the Gamma function, $K_\nu$ is the modified Bessel function of the second kind, $\nu>1/2$ is a smoothness parameter, $\sigma^2$ is the variance and $\lambda_{\mathrm{C}}$ is a length scale parameter. Unlike the analysis, we conduct numerical tests without the smallness assumption in order to show the flexibility of the QMC method in a wider range of situations. Numerical results for initial data satisfying the smallness assumption are similar and hence not presented. Consider the Mat\'ern covariance with the following parameter settings:
$\nu=2.5,1.75$, $\lambda_{\mathrm{C}}=1,0.1$, and $\sigma^2=1,0.25$.
We compute the eigenpairs of \eqref{cov} on a finer mesh $\mathcal{T}_{h/4}$ using conforming piecewise linear elements. The eigenpairs $(\mu_j,\xi_j)_{j=1}^s$ are used to construct $s$-term truncated KL expansions with $s=400$, cf. Fig. \ref{NT_fg1}. Additionally, we obtain the sequence $\boldsymbol{b}=\{b_j\}_{j\geq1}$ in \eqref{ass_indiv} by utilizing the eigenpairs $(\mu_j,\xi_j)$. The summability parameter $p$ in Assumption \ref{ass_indiv} is estimated using linear regression on the sequence $|\log b_j|$ against $\log j$ for $500\leq j\leq 1000$.

Next, we use the component-by-component algorithm \cite{qmc_ref_13} to determine the generating vector $\boldsymbol{z}\in \mathbb{Z}^s$  for the RSLR with $N\in \mathbb{N}$ sampling points. To this end, we construct the weighted Sobolev space $\mathcal{W}_s$ using the norm defined in \eqref{wss}. The weight parameters $\gamma_{\fuu}$ are chosen according to \eqref{final_weight}, and the weight function $\psi^2_j(y)$ is defined in \eqref{active_weight}. We select $a_j=\sqrt{\frac{2\lambda_*-1}{8\lambda_*}}$ based on \eqref{aj_choice}, where the value of $\lambda_\ast$ depends on the empirical estimate $p$.
%Different values of $\lambda_\ast$ are chosen depending on the different parameters.
We take $R=32$ independent random shifts, cf. \eqref{QMC_main}, where each sample $\boldsymbol{\Delta}_r$ for $1\leq r\leq R$ is uniformly distributed over $[0,1]^s$. Then we use the generator vector $\boldsymbol{z}\in \mathbb{Z}^s$ to generate the $N$ sampling points:
$(\frac{i\boldsymbol{z} }{N}+\boldsymbol{\Delta}_r)\bmod 1$  for all $1\leq i\leq N$.
For the given quantity of interest defined by the linear functional $\mathcal{G}\in(L^2(D)^2)'$, we compute the approximation $Q_r:=\mathcal{Q}_{s,N}(F^J_{s,h};\boldsymbol{\Delta_r})$ and its mean $\bar{Q}$ over $R$ random shifts, and obtain an unbiased estimator for the RMSE$_{\rm qmc}$ by
%\begin{equation}
%\label{standard_error}
$(\frac{1}{R}\frac{1}{R-1}\sum_{r=1}^R(Q_r-\bar{Q})^2)^{1/2}$,
%\approx\sqrt{\mathbb{E}^{\boldsymbol{\Delta}}\left[\left(\mathbb{E}\left[\mathcal{G}(\uu^J_{s,h})\right]-\mathcal{Q}_{s,N}(F^J_{s,h};\boldsymbol{\Delta})\right)^2\right]}.
%\end{equation}
i.e., the so-called standard error.
We choose two bounded linear functionals $\mathcal{G}_1$ and $\mathcal{G}_2$: $\mathcal{G}_1$ evaluates the first component at the point $(1/2, 1/2)$ for $t=0.1$, and $\mathcal{G}_2$ evaluates the second component at the same point for $t=0.2$. We compute the standard errors $e_1$ and $e_2$ for $\mathcal{G}_1$ and $\mathcal{G}_2$, respectively.

\begin{figure}[htp!]
\centering
\subfloat[$\|\xi_j\|_{\mathbb{C}(\overline{D})}$]{\label{fig:6.1}\includegraphics[trim={17cm 1cm 17cm 1.5cm},clip,width=1.7in]{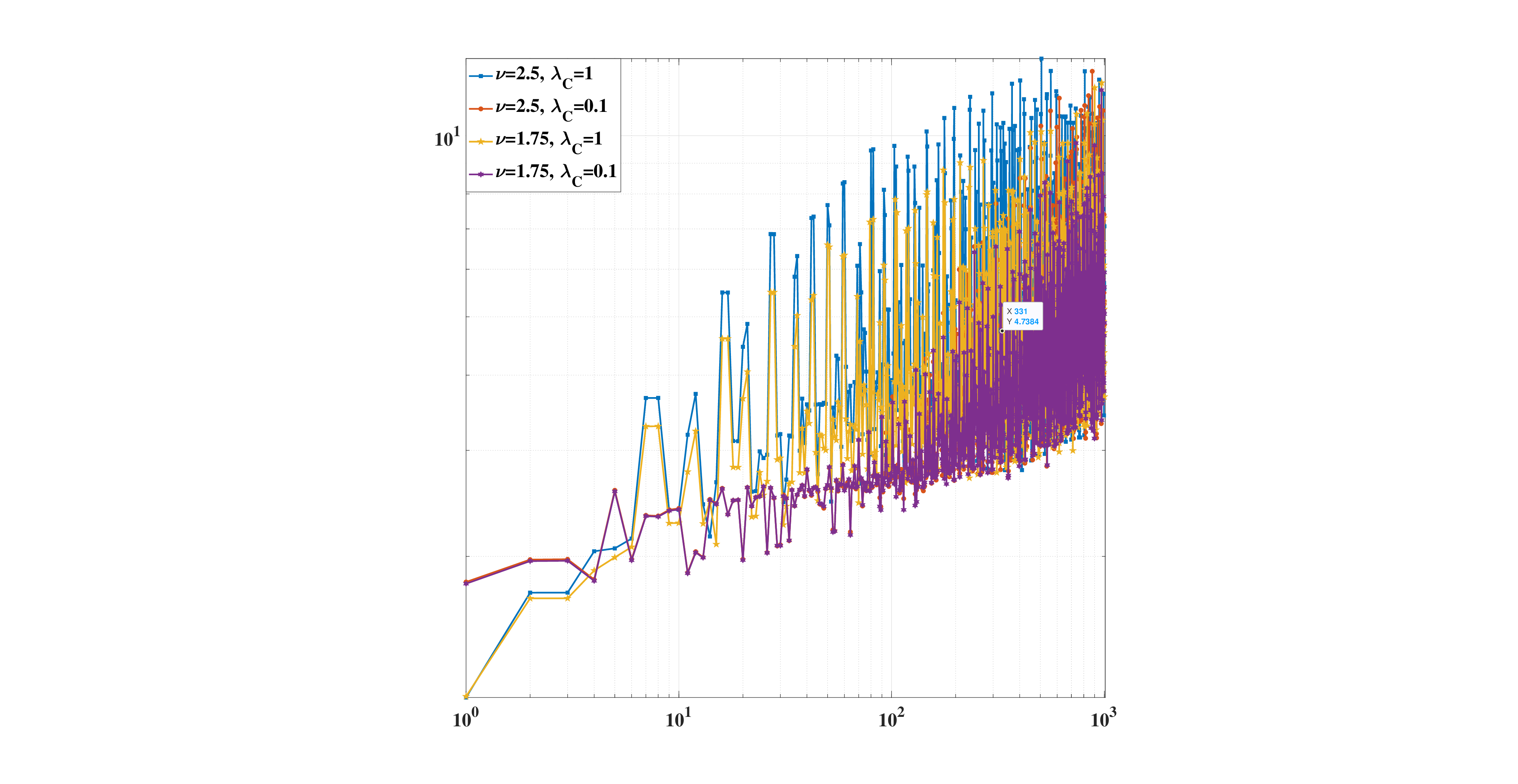}}
\subfloat[$\|\nabla \xi_j\|_{\mathbb{C}(\overline{D})}$]{\label{fig:6.2}\includegraphics[trim={17cm 1cm 17cm 1.5cm},clip,width=1.7in]{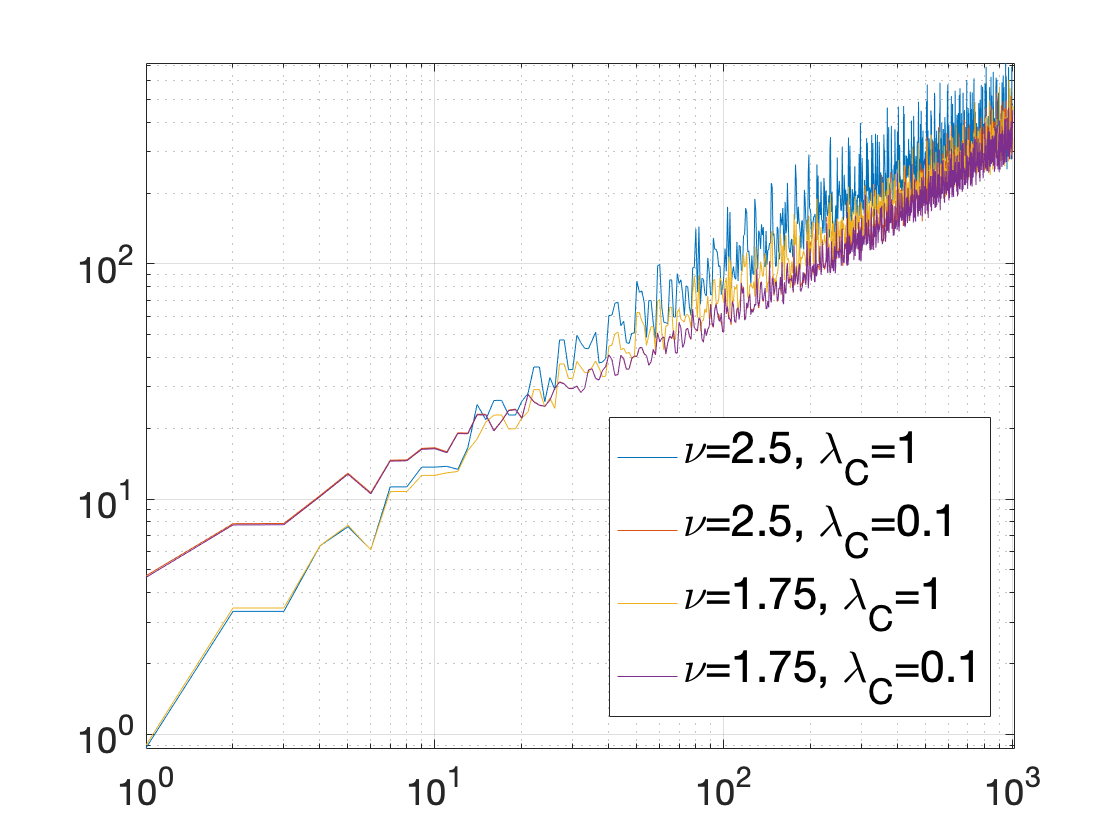}}
\subfloat[$b_j=\sqrt{\mu_j}||\xi_j||_{\mathbb{C}(\overline{D})}$]{\label{fig:6.3}
\includegraphics[trim={17cm 1cm 17cm 1.5cm},clip,width=1.7in]{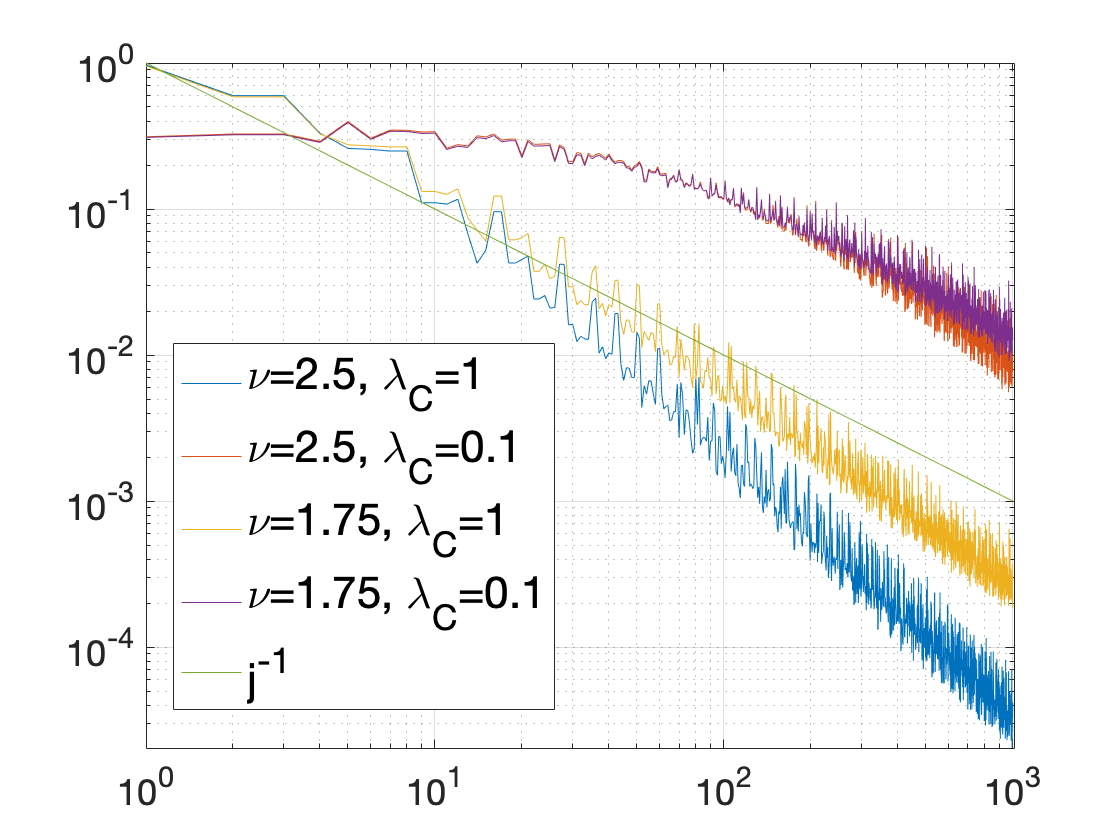}}
\caption{Log-log plot of $\|\xi_j\|_{\mathbb{C}(\overline{D})}$, $\|\nabla \xi_j\|_{\mathbb{C}(\overline{D})}$ and $b_j$ against $j$ for the Mat\'ern covariance with $\nu=2.5$, $\sigma^2=1$ and $\lambda_{\mathrm{C}}=1$.}
\label{NT_fg1}
\end{figure}
\begin{table}[htp!]
  \centering
\caption{Comparison of the standard error of QMC and MC with $\mathcal{G}_1$ and $\mathcal{G}_2$ for Mat\'ern covariance with $\nu=2.5$, $\lambda_{\mathrm{C}}=1$ and different $\sigma^2$.\label{Numerical_Table_1}}
\begin{tabular}{c|cccc|cccc|}
\toprule
    \multicolumn{1}{c}{}& \multicolumn{4}{c}{$\sigma^2=1$}&\multicolumn{4}{c}{$\sigma^2=0.25$}\\
    \cmidrule(lr){2-5} \cmidrule(lr){6-9}
    \multicolumn{1}{c}{$N$} & \multicolumn{2}{c}{QMC} & \multicolumn{2}{c}{MC} & \multicolumn{2}{c}{QMC} & \multicolumn{2}{c}{MC}  \\
    \cmidrule(lr){2-3} \cmidrule(lr){4-5} \cmidrule(lr){6-7} \cmidrule(lr){8-9}
    & $e_1$ &$e_2$  & $e_1$   & $e_2$   & $e_1$ & $e_2$ & $e_1$ & $e_2$  \\
 \midrule
1009               & 6.36e-4 & 4.65e-5 & 1.78e-3 & 1.68e-4 & 7.90e-5 & 8.60e-6 & 4.19e-4 & 3.23e-5 \\
2003               & 3.77e-4 & 6.00e-5 & 1.61e-3 & 1.14e-4 & 2.92e-5 & 4.53e-6 & 2.99e-4 & 2.06e-5 \\
4001               & 3.79e-4 & 5.80e-5 & 9.21e-4 & 1.17e-4 & 2.85e-5 & 4.17e-6 & 2.27e-4 & 1.77e-5 \\
8009               & 1.79e-4 & 3.03e-5 & 5.76e-4 & 5.22e-5 & 1.04e-5 & 2.18e-6 & 1.22e-4 & 1.30e-5 \\
16001              & 1.23e-4 & 1.12e-5 & 4.21e-4 & 3.34e-5 & 1.32e-5 & 8.40e-7 & 8.50e-5 & 6.66e-6 \\
32003              & 6.24e-5 & 1.20e-5 & 2.32e-4 & 2.24e-5 & 8.60e-6 & 7.55e-7 & 6.05e-5 & 5.05e-6 \\
64007              & 4.31e-5 & 7.90e-6 & 1.81e-4 & 1.79e-5 & 2.25e-6 & 1.76e-7 & 4.84e-5 & 4.55e-6 \\
\midrule
Rate               & 0.66     & 0.52     & 0.59     & 0.57     & 0.72     & 0.87     & 0.55     & 0.50     \\
\bottomrule
    \end{tabular}
\end{table}

First we take Mat\'ern covariance with a smoothness parameter $\nu=2.5$, a length scale parameter $\lambda_{\mathrm{C}}=1$ and a variance of $\sigma^2=1$ or $\sigma^2=0.25$. Numerically we observe that the sequence $\{b_j\}_{j=1}^{\infty}$ lies under the sequence $\{j^{-3/2}\}_{j=1}^{\infty}$ for sufficiently large $j$. Therefore, empirically, Assumption \ref{ass_indiv} holds with some $p\leq 2/3$, and thus we take $\lambda_\ast:=0.55$ for these two cases. In Table \ref{Numerical_Table_1}, we compare the standard error of QMC with that of MC for the linear functionals $\mathcal{G}_1$ and $\mathcal{G}_2$. To obtain more precise estimates, for each number of sampling points $N$, we use the mean of ten different tests with different uniformly distributed {\textit{random shift}} ${\boldsymbol{\Delta}}$ in \eqref{QMC_main}. The rate of convergence is estimated by performing a linear regression of the negative logarithm of the standard error against $\log N$, based on the mean of the ten tests. The results show that the QMC method yields a smaller error and a faster convergence rate for both cases. For instance, when the number of sampling points is $N=64007$ and the variance is $\sigma^2=1$, the standard errors for the QMC and MC methods are 4.31e-5 and 1.81e-4, respectively.
For the case of $\lambda_{\mathrm{C}}=0.1$, the empirical parameter $p$ has the value of $0.6832$, and hence we choose $\lambda_*=0.52$, which is presented in Table \ref{Numerical_Table_2}. The performance of the QMC method is similar to that observed in Table \ref{Numerical_Table_1}.
%However, due to computational limitations, we are unable to compute results with a larger number of sampling points beyond $N=64007$.
%It is worth noting that a larger number of sampling points could potentially yield a better and more accurate convergence rate.

\begin{table}[htp!]
  \centering
\caption{Comparison of the standard error of QMC and MC with $\mathcal{G}_1$ and $\mathcal{G}_2$ for Mat\'ern covariance with $\nu=2.5$, $\lambda_{\mathrm{C}}=0.1$ and different $\sigma^2$.\label{Numerical_Table_2}}
\begin{tabular}{c|cccc|cccc|}
\toprule
    \multicolumn{1}{c}{}& \multicolumn{4}{c}{$\sigma^2=1$}&\multicolumn{4}{c}{$\sigma^2=0.25$}\\
    \cmidrule(lr){2-5} \cmidrule(lr){6-9}
    \multicolumn{1}{c}{$N$} & \multicolumn{2}{c}{QMC} & \multicolumn{2}{c}{MC} & \multicolumn{2}{c}{QMC} & \multicolumn{2}{c}{MC}  \\
    \cmidrule(lr){2-3} \cmidrule(lr){4-5} \cmidrule(lr){6-7} \cmidrule(lr){8-9}
    & $e_1$ &$e_2$  & $e_1$   & $e_2$   & $e_1$ & $e_2$ & $e_1$ & $e_2$  \\
 \midrule
1009               & 1.76e-3 & 8.26e-5 & 1.92e-3 & 1.68e-4 & 1.56e-4 & 1.42e-5 & 6.09e-4 & 5.79e-5 \\
2003               & 1.14e-3 & 6.82e-5 & 3.81e-3 & 2.62e-4 & 8.88e-5 & 1.24e-5 & 1.22e-3 & 8.23e-5 \\
4001               & 4.00e-4 & 3.80e-5 & 3.37e-3 & 1.33e-4 & 6.66e-5 & 6.85e-6 & 1.00e-3 & 3.87e-5 \\
8009               & 3.91e-4 & 2.61e-5 & 1.33e-3 & 1.31e-4 & 3.07e-5 & 2.07e-6 & 4.20e-4 & 3.51e-5 \\
16001              & 2.18e-4 & 1.96e-5 & 1.30e-3 & 6.77e-5 & 2.12e-5 & 1.81e-6 & 3.31e-4 & 2.27e-5 \\
32003              & 1.66e-4 & 1.03e-5 & 7.99e-4 & 5.10e-5 & 2.21e-5 & 1.68e-6 & 2.63e-4 & 1.67e-5 \\
64007              & 9.53e-5 & 8.67e-6 & 7.14e-4 & 4.33e-5 & 7.97e-6 & 7.72e-7 & 2.33e-4 & 1.33e-5 \\
\midrule
Rate               & 0.68     & 0.58     & 0.36     & 0.41     & 0.66     & 0.73     & 0.36     & 0.42     \\
\bottomrule
    \end{tabular}
\end{table}
\begin{table}[htp!]
  \centering
    \caption{Comparison of the standard error of QMC and MC with $\mathcal{G}_1$ and $\mathcal{G}_2$ for Mat\'ern covariance with $\nu=1.75$, $\lambda_{\mathrm{C}}=1$ and different $\sigma^2$.\label{Numerical_Table_3}}
    \begin{tabular}{c|cccc|cccc|}
    \toprule
    \multicolumn{1}{c}{}& \multicolumn{4}{c}{$\sigma^2=1$}&\multicolumn{4}{c}{$\sigma^2=0.25$}\\
    \cmidrule(lr){2-5} \cmidrule(lr){6-9}
    \multicolumn{1}{c}{$N$} & \multicolumn{2}{c}{QMC} & \multicolumn{2}{c}{MC} & \multicolumn{2}{c}{QMC} & \multicolumn{2}{c}{MC}  \\
    \cmidrule(lr){2-3} \cmidrule(lr){4-5} \cmidrule(lr){6-7} \cmidrule(lr){8-9}
    & $e_1$ &$e_2$  & $e_1$   & $e_2$   & $e_1$ & $e_2$ & $e_1$ & $e_2$  \\
 \midrule
1009               & 7.09e-4 & 1.06e-4 & 1.61e-3 & 2.35e-4 & 7.32e-5 & 1.14e-5 & 3.66e-4 & 4.96e-5 \\
2003               & 4.69e-4 & 5.47e-5 & 1.59e-3 & 1.42e-4 & 7.53e-5 & 6.66e-6 & 3.57e-4 & 2.79e-5 \\
4001               & 2.93e-4 & 3.85e-5 & 7.29e-4 & 8.35e-5 & 2.42e-5 & 3.01e-6 & 1.91e-4 & 1.99e-5 \\
8009               & 1.78e-4 & 2.85e-5 & 5.14e-4 & 7.16e-5 & 1.97e-5 & 1.68e-6 & 1.39e-4 & 1.65e-5 \\
16001              & 1.91e-4 & 1.78e-5 & 5.12e-4 & 5.37e-5 & 1.16e-5 & 1.48e-6 & 1.14e-4 & 1.26e-5 \\
32003              & 1.15e-4 & 1.03e-5 & 5.90e-4 & 3.70e-5 & 8.98e-6 & 6.44e-7 & 1.27e-4 & 7.45e-6 \\
64007              & 6.77e-5 & 6.08e-6 & 3.29e-4 & 2.67e-5 & 4.34e-6 & 5.32e-7 & 8.80e-5 & 8.31e-6 \\
\midrule
Rate               & 0.53     & 0.65     & 0.37     & 0.50     & 0.69     & 0.75     & 0.35     & 0.44     \\
\bottomrule
    \end{tabular}
\end{table}

Next, we consider Mat\'ern covariance with a smoothness parameter $\nu=1.75$ and length scale parameter $\lambda_{\mathrm{C}}=1$ or $\lambda_{\mathrm{C}}=0.1$, with variances of $\sigma^2=1$ or $\sigma^2=0.25$, respectively. For $\lambda_{\mathrm{C}}=1$, the empirical parameter $p$ is estimated to be $0.7198$, leading to the choice of $\lambda_\ast=0.56$. For $\lambda_{\mathrm{C}}=0.1$, the empirical parameter $p$ is estimated to be $0.8988$, and accordingly, we choose $\lambda_\ast=0.81$. The numerical results of the QMC scheme using these parameters are presented in Tables \ref{Numerical_Table_3} and \ref{Numerical_Table_4}.

\begin{table}[htp!]
  \centering
\caption{Comparison of the standard error of QMC and MC with $\mathcal{G}_1$ and $\mathcal{G}_2$ for Mat\'ern covariance with $\nu=1.75$, $\lambda_{\mathrm{C}}=0.1$ and different $\sigma^2$.\label{Numerical_Table_4}}
    \begin{tabular}{c|cccc|cccc|}
    \toprule
    \multicolumn{1}{c}{}& \multicolumn{4}{c}{$\sigma^2=1$}&\multicolumn{4}{c}{$\sigma^2=0.25$}\\
    \cmidrule(lr){2-5} \cmidrule(lr){6-9}
    \multicolumn{1}{c}{$N$} & \multicolumn{2}{c}{QMC} & \multicolumn{2}{c}{MC} & \multicolumn{2}{c}{QMC} & \multicolumn{2}{c}{MC}  \\
    \cmidrule(lr){2-3} \cmidrule(lr){4-5} \cmidrule(lr){6-7} \cmidrule(lr){8-9}
    & $e_1$ &$e_2$  & $e_1$   & $e_2$   & $e_1$ & $e_2$ & $e_1$ & $e_2$  \\
 \midrule
1009               & 1.76e-3 & 9.78e-5 & 2.73e-3 & 2.01e-4 & 2.12e-4 & 9.82e-6 & 8.70e-4 & 6.65e-5 \\
2003               & 7.76e-4 & 7.40e-5 & 3.24e-3 & 2.57e-4 & 1.19e-4 & 1.03e-5 & 1.14e-3 & 8.00e-5 \\
4001               & 3.78e-4 & 3.73e-5 & 3.30e-3 & 9.66e-5 & 5.05e-5 & 3.67e-6 & 1.02e-3 & 3.33e-5 \\
8009               & 4.54e-4 & 3.60e-5 & 1.11e-3 & 9.13e-5 & 6.07e-5 & 2.68e-6 & 3.74e-4 & 3.04e-5 \\
16001              & 9.48e-5 & 1.50e-5 & 8.76e-4 & 6.76e-5 & 1.67e-5 & 2.03e-6 & 2.42e-4 & 2.14e-5 \\
32003              & 9.23e-5 & 8.54e-6 & 8.32e-4 & 5.82e-5 & 1.36e-5 & 7.56e-7 & 2.65e-4 & 1.76e-5 \\
64007              & 1.11e-4 & 6.32e-6 & 5.22e-4 & 3.62e-5 & 7.13e-6 & 6.23e-7 & 1.95e-4 & 1.21e-5 \\
\midrule
Rate               & 0.72     & 0.69     & 0.47     & 0.44     & 0.81     & 0.73     & 0.46     & 0.44     \\
\bottomrule
    \end{tabular}
\end{table}

Finally, we summarize all the numerical experiments from Table \ref{Numerical_Table_1} to Table \ref{Numerical_Table_4} and illustrate their convergence in Fig. \ref{NT_fg2-1} and Fig. \ref{NT_fg2-2}. The results clearly demonstrate that the proposed QMC method outperforms the standard MC method in terms of the standard error for all cases. Additionally, the QMC method exhibits a faster convergence rate compared to the MC method in most cases. These findings validate the effectiveness of the proposed QMC approach for solving the given problem.

\begin{figure}[htp!]
\centering
\subfloat[]{\includegraphics[trim={15cm 0.5cm 17cm 0.5cm},clip,width=2.3in]{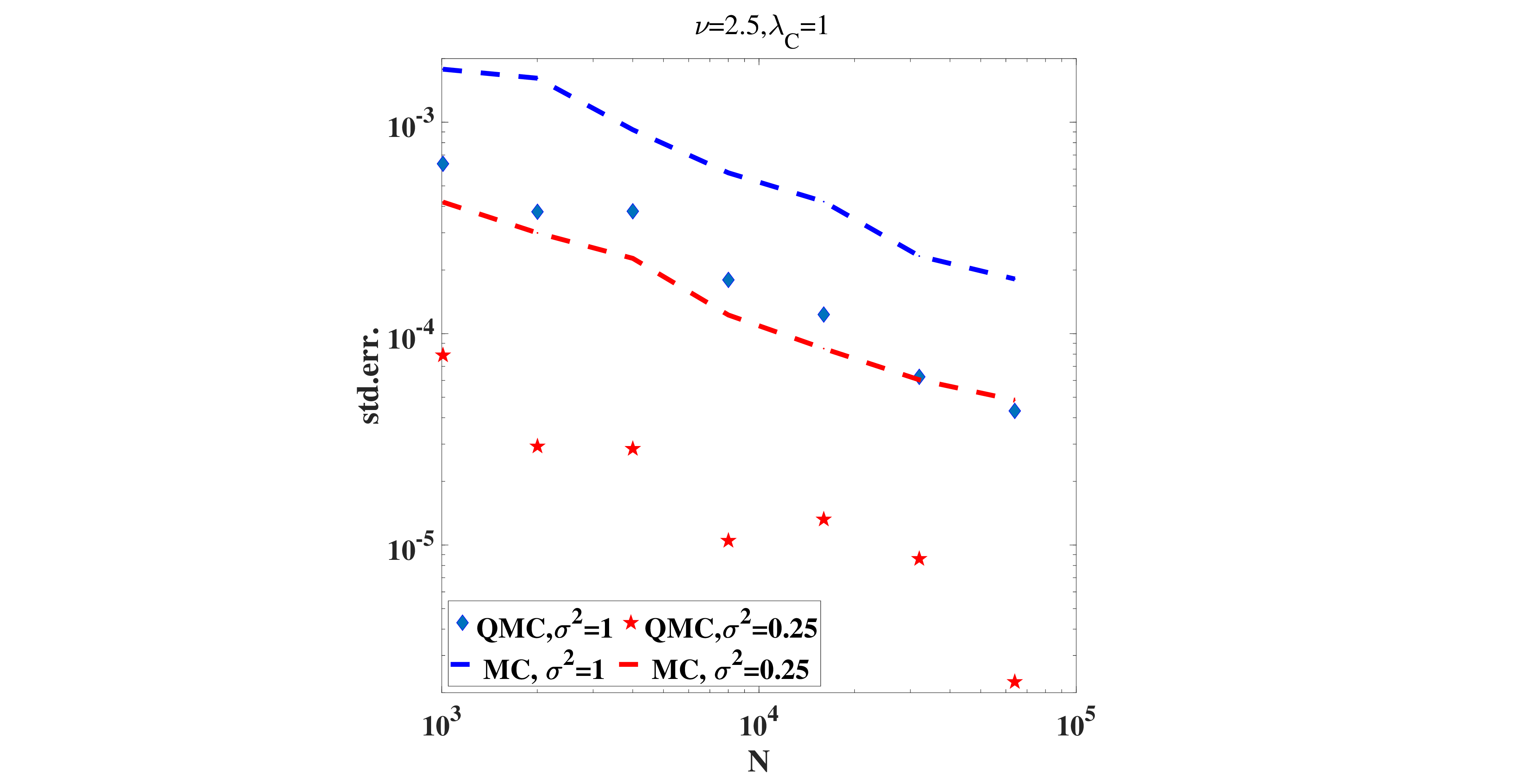}}
\subfloat[]{\includegraphics[trim={15cm 0.5cm 17cm 0.5cm},clip,width=2.3in]{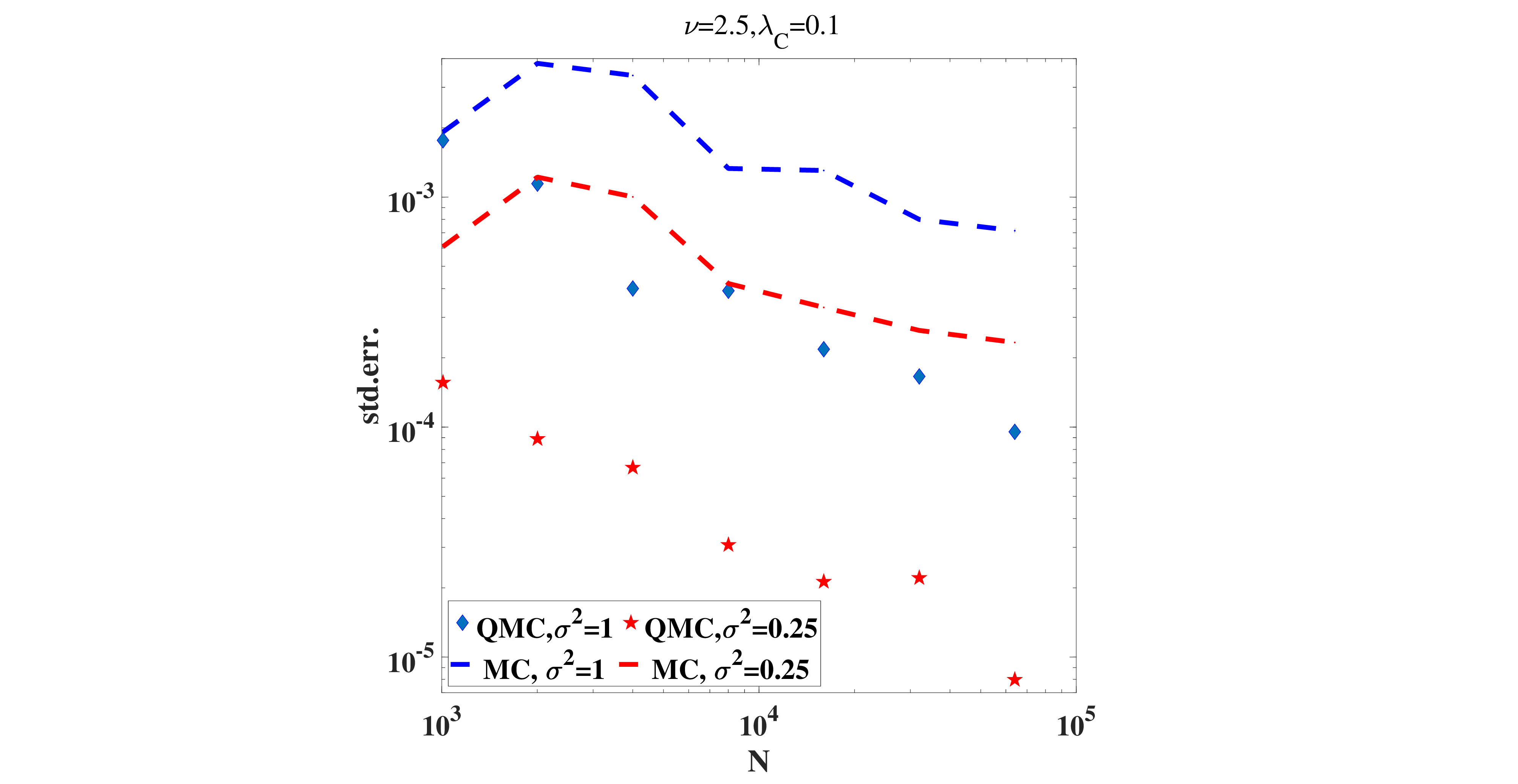}}\\
\subfloat[]{
\includegraphics[trim={15cm 0.5cm 17cm 0.5cm},clip,width=2.3in]{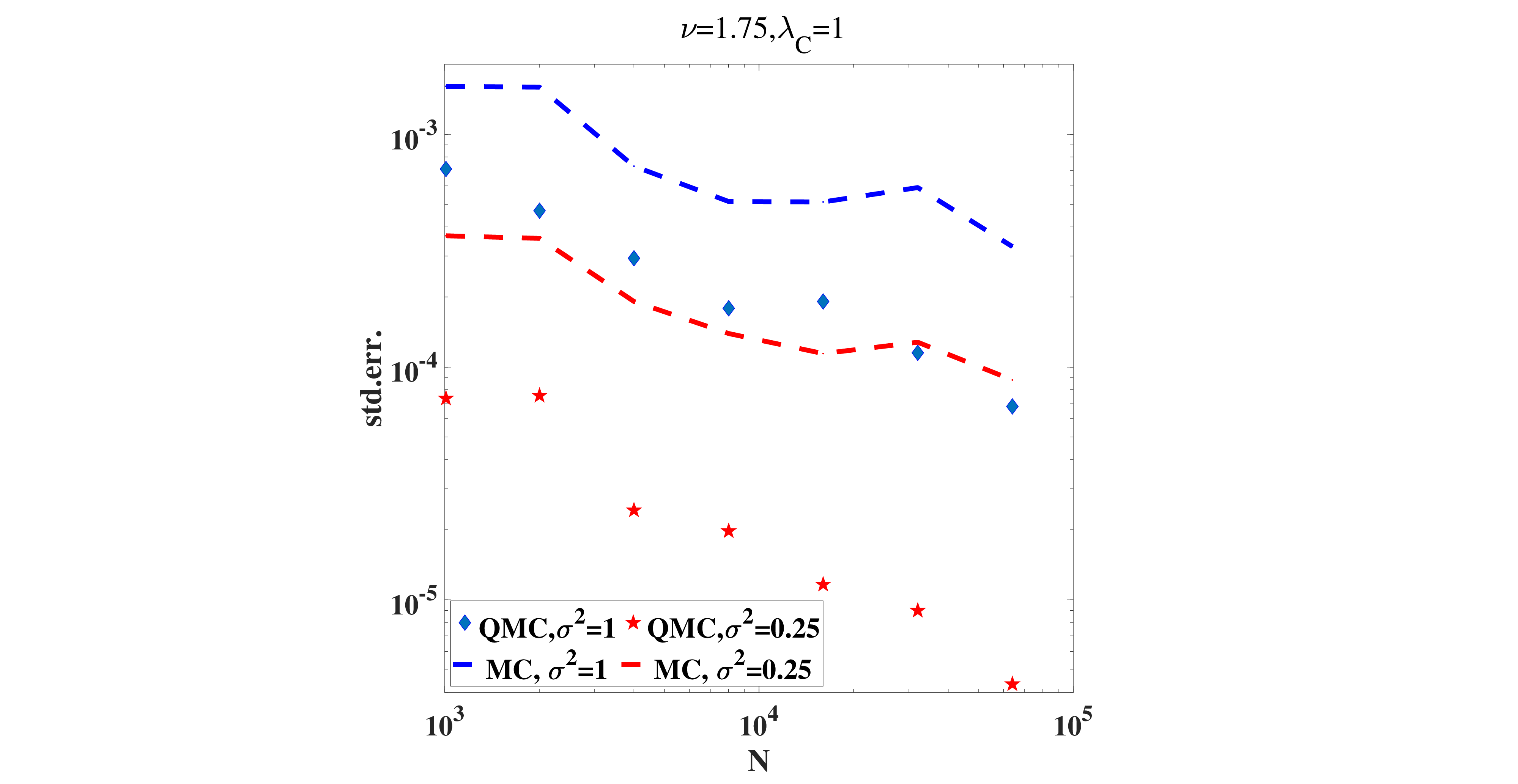}}
\subfloat[]{
\includegraphics[trim={15cm 0.5cm 17cm 0.5cm},clip,width=2.3in]{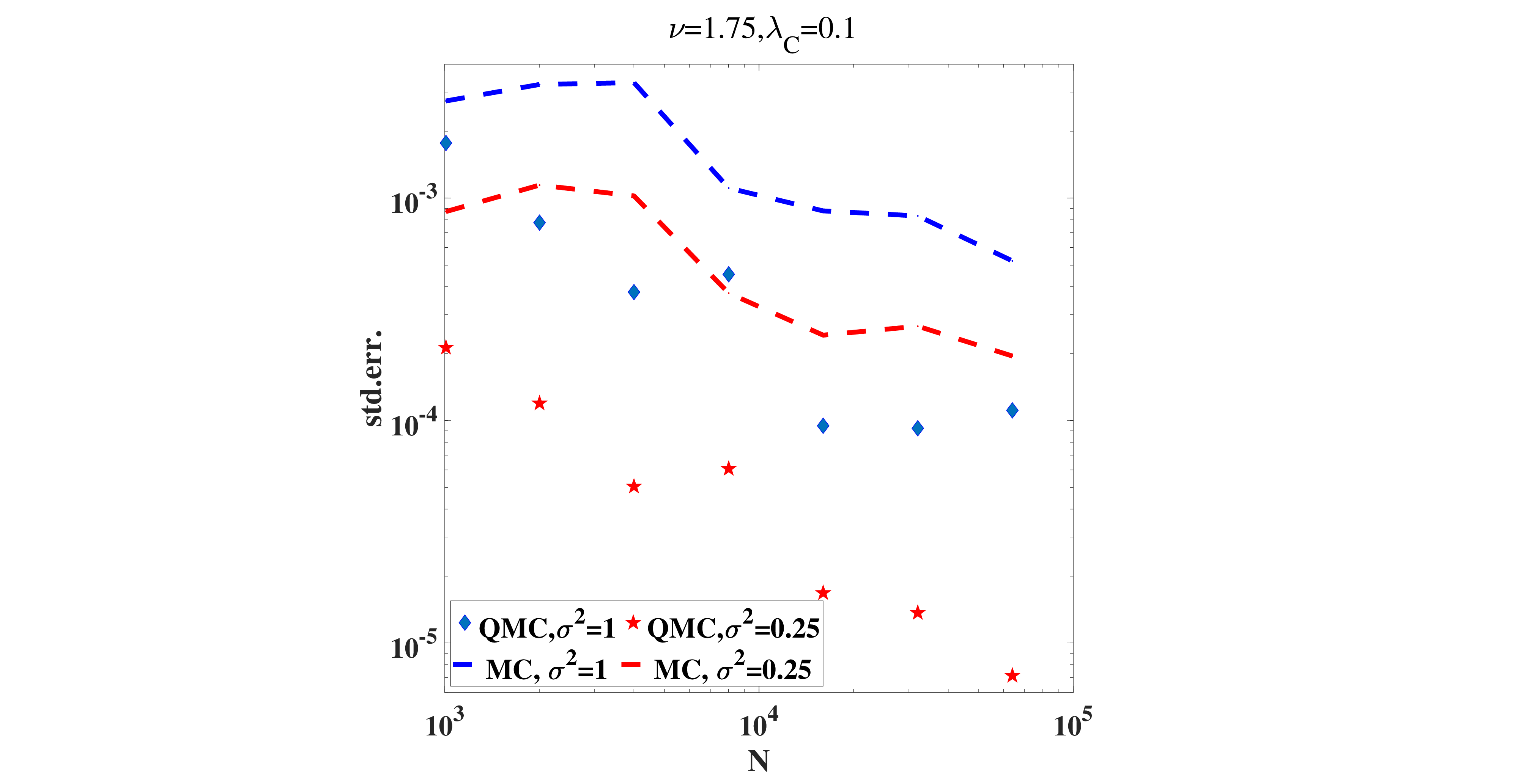}}
\caption{The standard errors of $e_1$ with various Mat\'ern covariance parameters for QMC and MC plotted versus the number of sampling points $N$.}
\label{NT_fg2-1}
\end{figure}

\begin{figure}[htp!]
\centering
\subfloat[]{\includegraphics[trim={15cm 0.5cm 17cm 0.5cm},clip,width=2.3in]{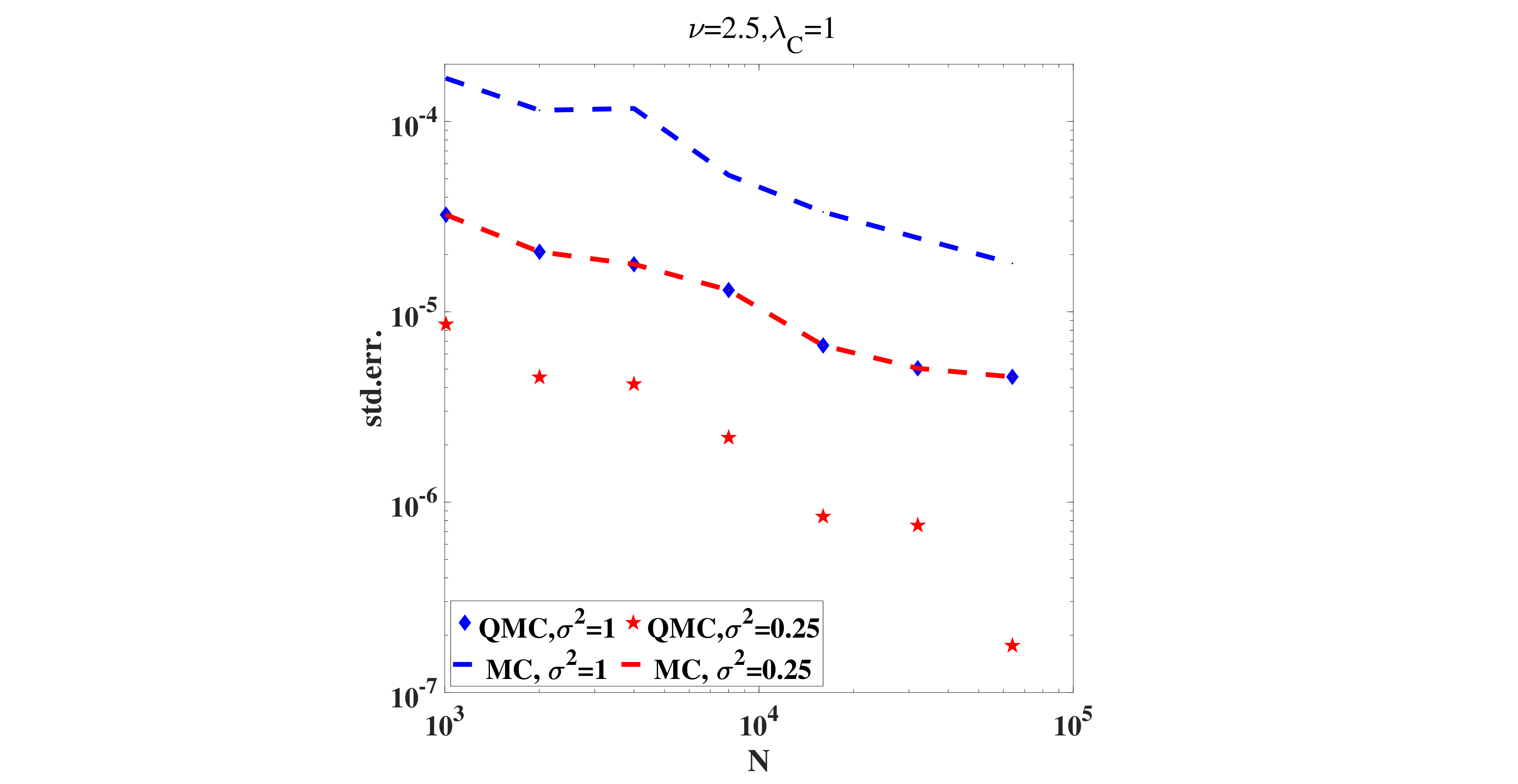}}
\subfloat[]{\includegraphics[trim={15cm 0.5cm 17cm 0.5cm},clip,width=2.3in]{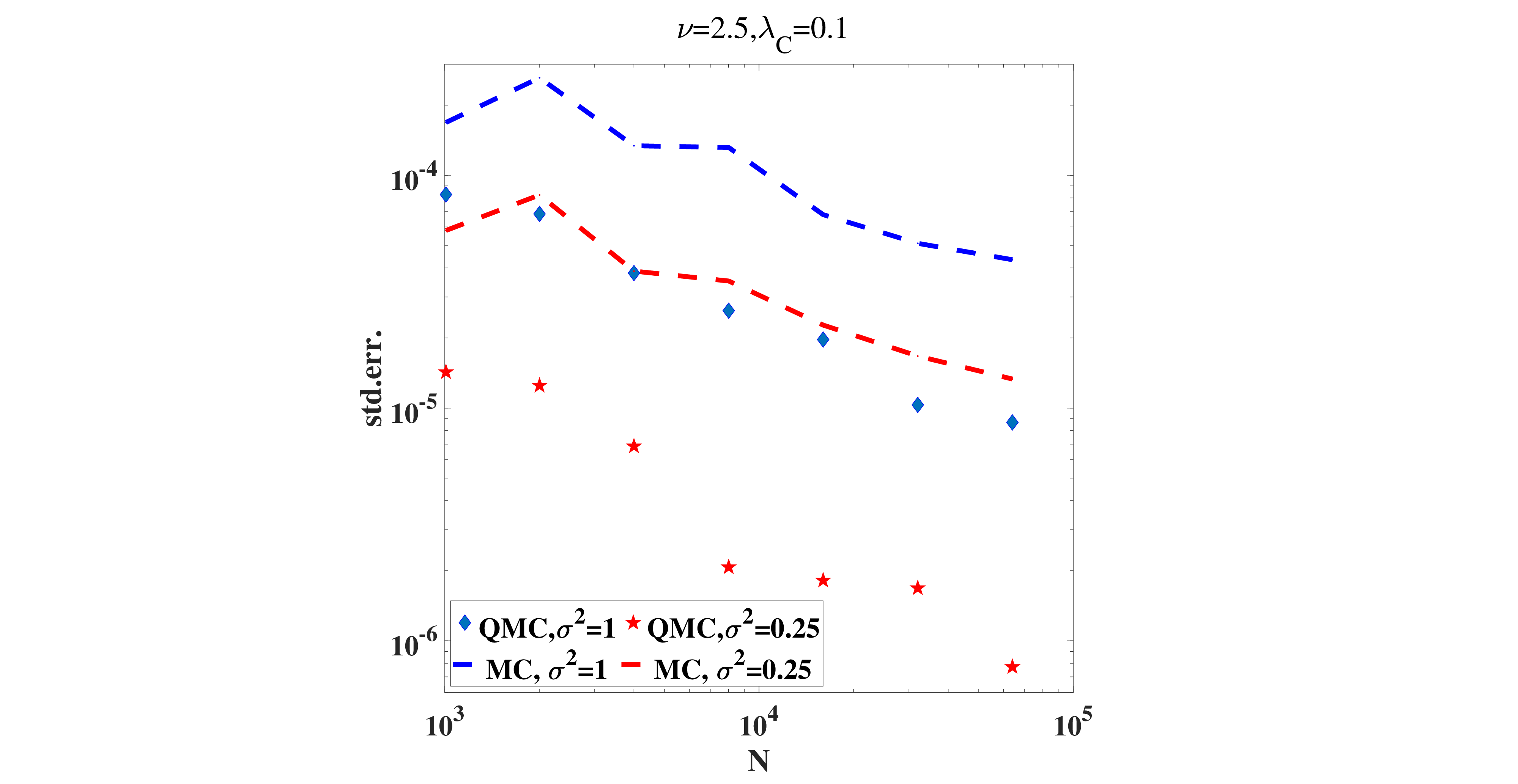}}\\
\subfloat[]{
\includegraphics[trim={15cm 0.5cm 17cm 0.5cm},clip,width=2.3in]{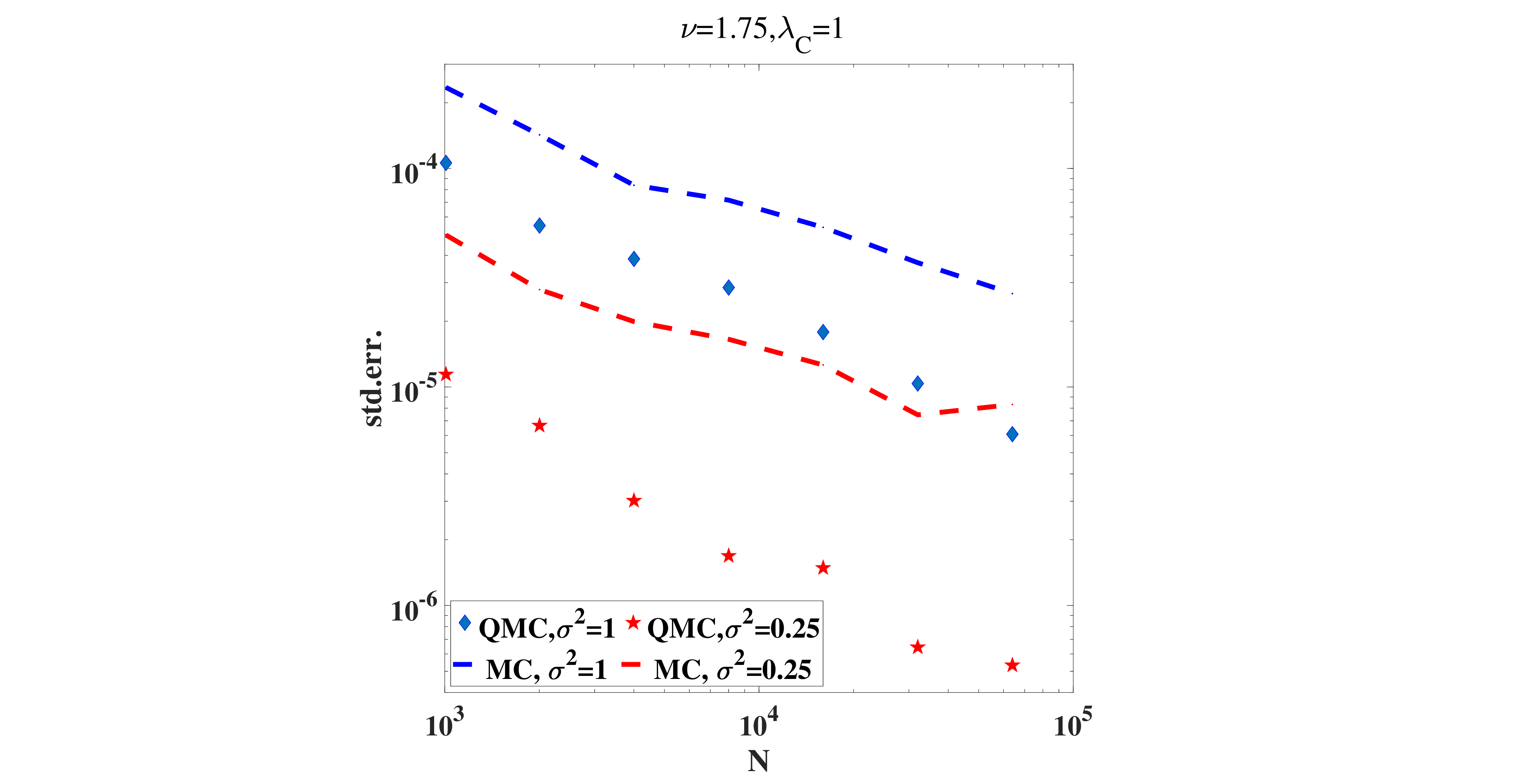}}
\subfloat[]{
\includegraphics[trim={15cm 0.5cm 17cm 0.5cm},clip,width=2.3in]{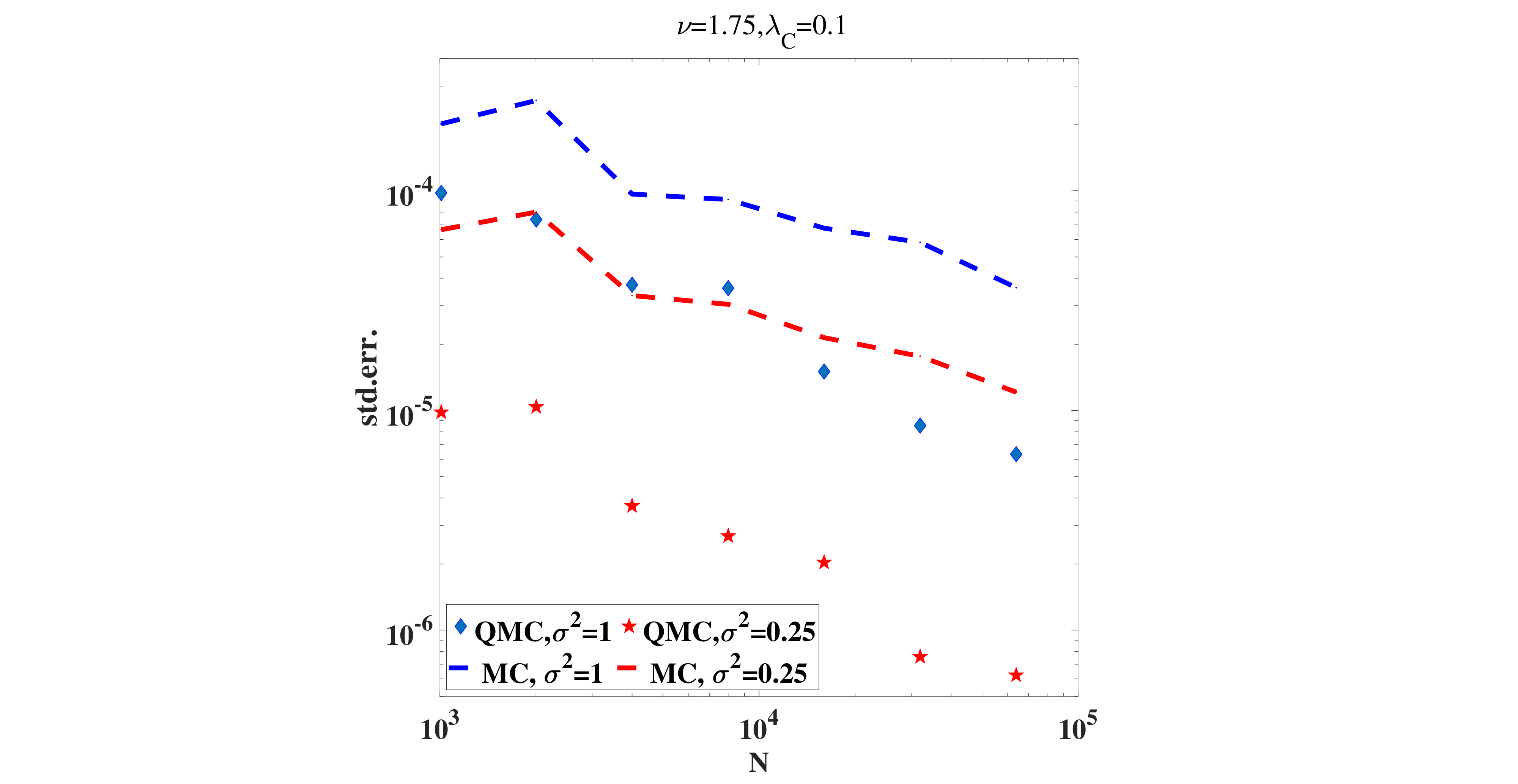}}
\caption{The standard errors of $e_2$ with various Mat\'ern covariance parameters for QMC and MC plotted versus the number of sampling points $N$.}
\label{NT_fg2-2}
\end{figure}

{
Theorem \ref{QMC_final} suggests that the asymptotic behavior of the $\text{RMSE}_{\text{qmc}}$ depends on $\chi$ (and hence $\lambda_*$). In practice, the empirical convergence rates may be lower than the theoretical prediction. This is indeed observed in Figs. \ref{NT_fg2-1} and \ref{NT_fg2-2}. Note that even for a linear PDE, it is not always possible to observe the optimal convergence rate. For example, in the work \cite{main_ref}, the authors also obtained worse QMC convergence rates of 0.55 and 0.58 for $\sigma^2 = 4$, which are lower than that for the MC (0.55 and 0.63). In our experiment, this is attributed to the fact that we have chosen a relatively small random shift $R = 32$ and a small sampling point $N$ (to avoid the high computational complexity). It is expected that a larger $R$ and a larger $N$ can lead to a better convergence rate. In view of Theorem \ref{QMC_final}, the convergence rate $\chi$ depends on the exponent $p\in(0,1)$ in Assumption \ref{ass_indiv}. A larger $p$ implies a slower convergence rate, and vice versa. The value of the parameter $p$ can be estimated in Fig. \ref{NT_fg1}(c), for each choice of parameters $\nu$ and $\lambda_C$, which can potentially be inaccurate. These brief discussions highlight the delicacy of attaining the theoretical convergence rates for QMC methods.}

\section{Conclusion}\label{sec_con}

In this paper, we have investigated the Navier-Stokes equations with random initial data in a
bounded polygonal domain $D\subset\mathbb{R}^2$. We have developed a scheme for
approximating the expected value of the solution, by combining Galerkin FEM, truncated Karhunen-Lo\`{e}ve
expansion of the log-normal initial random field, and quasi-Monte Carlo (QMC) method. Further, we
have derived a bound on the root-mean-square error, including the finite element error, the dimension
truncation error, and the error from the QMC quadrature. The numerical experiments show that
the scheme enjoys fast convergence with respect to the number of sampling points.
{{The present investigation indicates that the QMC methods have enormous potentials for solving nonlinear PDEs with uncertainties.}}

Theoretically, it is of interest to extend the analysis to a more general class of PDEs. Of particular interest are non-Newtonian fluid flow models, describing the motion of fluids with a more general structure \cite{Malek, Malek_NS}. The extension requires a new way to control the fully non-linear diffusion term, and a more refined analysis for the convective term. {Of equal interest is to relax the smallness condition, which plays a crucial role in the overall analysis. Numerically the condition does not appear very critical for the empirical performance.} From an algorithmic point of view, it is of much interest to consider the multilevel and/or changing dimension algorithms \cite{multi_1, multi_2, multi_3, multi_4}, with the QMC algorithm applied to the PDE problems with random data \cite{QMC_multi_4, QMC_multi_3, QMC_multi_2, QMC_multi_1}. Adapting these algorithms to the Navier--Stokes problem is an intriguing topic.

%\bmhead{Supplementary information}

%If your article has accompanying supplementary file/s please state so here.

%Authors reporting data from electrophoretic gels and blots should supply the full unprocessed scans for key as part of their Supplementary information. This may be requested by the editorial team/s if it is missing.

%Please refer to Journal-level guidance for any specific requirements.

%\bmhead{Acknowledgments}
%Acknowledgments are not compulsory. Where included they should be brief. Grant or contribution numbers may be acknowledged.
%Please refer to Journal-level guidance for any specific requirements.
%\section*{Declarations}
%Some journals require declarations to be submitted in a standardised format. Please check the Instructions for Authors of the journal to which you are submitting to see if you need to complete this section. If yes, your manuscript must contain the following sections under the heading `Declarations':

%\begin{itemize}
%\item Conflict of interest/Competing interests (check journal-specific guidelines for which heading to use)
%\item Authors' contributions
%\end{itemize}

%\noindent
%If any of the sections are not relevant to your manuscript, please include the heading and write `Not applicable' for that section.

\section*{Acknowledgements}
SK was supported by and National Research Foundation of
Korea Grant funded by the Korean Government (RS-2023-00212227).
GL acknowledges the support from Young Scientists fund (Project number: 12101520) by NSFC and GRF (project number: 17317122), RGC, Hong Kong. We thank the participants of HDA2023 for many constructive advice, especially Dr. Alexander Gilbert (UNSW, Sydney) and Prof. Ian Sloan (UNSW, Sydney).

\begin{appendices}

\section{Useful inequalities}
In this appendix, we collect several useful inequalities.
We use Ladyzhenskaya's inequality \cite[Lemma 3.3]{Temam} frequently.
\begin{lemma}\label{2d_est}
For any open set $D\subset\R^2$, the following inequality holds
\begin{equation}\label{2d_ineq}
\|\uu\|_{L^4(D)}\leq 2^{\frac{1}{4}}\|\uu\|^{\frac{1}{2}}_{L^2(D)}\|\nabla\uu\|^{\frac{1}{2}}_{L^2(D)}, \quad\forall\uu\in H^1_0(D)^2.
\end{equation}
\end{lemma}

We also need the following discrete Gronwall inequality \cite[Lemma 5.1]{heywood4}.
\begin{lemma}\label{gron}
Let the non-negative numbers $k,B$ and $a_j$, $b_j$, $c_j$, $\gamma_j$ satisfy
$a_n+k\sum^n_{j=0}b_j\leq k\sum^n_{j=0}\gamma_j a_j+k\sum^n_{j=0}c_j+B$, for all $n\geq 0.$
If $k\gamma_j<1$ for all $j\geq0$, then with $\sigma_j:=(1-k\gamma_j)^{-1}$
\[
a_n+k\sum^n_{j=0}b_j\leq\bigg(k\sum^n_{j=0}c_j+B\bigg)\exp\bigg(k\sum^n_{j=0}\sigma_j\gamma_j\bigg),\quad\forall n\geq0.
\]
\end{lemma}

We also need the following version of Fernique's theorem.
\begin{theorem}\label{Fernique}
 Let $E$ be a real separable Banach space and $X$ be an $E$-valued and centered Gaussian random variable, in the sense that, for each $x^*\in E^*$, $\langle X,x^*\rangle$ is a centered, real-valued Gaussian random variable. Then with $R\defeq\inf\{r\in[0,\infty):\mathbb{P}(\|X\|_E\leq r)\geq\frac{3}{4}\}$,
\[
\int_{\Omega}\exp\bigg(\frac{\|X\|^2_E}{18R^2}\bigg)\,{\rm{d}}\mathbb{P}(\omega)\lesssim 1.
\]
\end{theorem}

Last, we recall the following two auxiliary lemmas \cite{KSS2012, main_ref}.
\begin{lemma}\label{aux_conv_lemma_1}
Fix $m\in\mathbb{N}$, $\lambda>0$ and $A_i$, $B_i>0$ for all $i\in\mathbb{N}$. Then the quantity
\begin{equation}\label{conv_quant}
\left(\sum_{i=1}^mx^{\lambda}_{i}A_i\right)^{\frac{1}{\lambda}}\left(\sum_{i=1}^m\frac{B_i}{x_i}\right)
\end{equation}
is minimized over any sequences $(x_i)_{1\leq i\leq m}$ when
\begin{equation}\label{min_def}
x_i=c\left({B_i/A_i}\right)^{\frac{1}{1+\lambda}}\quad{\rm{for}}\,\,{\rm{all}}\,\,c>0.
\end{equation}
For $m\rightarrow\infty$, the function \eqref{conv_quant} is minimized provided that $x_i$ is defined by \eqref{min_def} for each $i$ and it  is finite if and only if $\sum_{i=1}^{\infty}(A_i B_i^{\lambda})^{1/(1+\lambda)}$ converges.
\end{lemma}

\begin{lemma}\label{aux_conv_lemma_2}
Fix $A_j>0$ for all $j\in\mathbb{N}$ and $\sum_{j\geq1}A_j<1$. Then we have
\[
\sum_{|\fuu|<\infty}|\fuu|!\prod_{j\in\fuu}A_j\leq\sum_{k=0}^{\infty}\bigg(\sum_{j\geq1}A_j\bigg)^k=\frac{1}{1-\sum_{j\geq1}A_j}.
\]
Furthermore, for any $B_j>0$ with $\sum_{j\geq1}B_j<\infty$, we also have
\[
\sum_{|\fuu|<\infty}\prod_{j\in\fuu}B_j=\prod_{j\geq1}(1+B_j)=\exp\bigg(\sum_{j\geq1}\log(1+B_j)\bigg)\leq\exp\bigg(\sum_{j\geq1}B_j\bigg).
\]
\end{lemma}

\end{appendices}


\bibliography{references}
\bibliographystyle{abbrv}

\end{document}